\documentclass[12pt]{amsart}
\usepackage{amsmath,amsthm,amscd,amsfonts, amssymb}
\usepackage{color}
\newcommand{\sect}[1]{\section{#1}\setcounter{equation}{0}}
\newcommand{\subsect}[1]{\subsection{#1}}

\font\mbn=msbm10 scaled \magstep1
\font\mbs=msbm7 scaled \magstep1
\font\mbss=msbm5 scaled \magstep1
\newfam\mbff
\textfont\mbff=\mbn
\scriptfont\mbff=\mbs
\scriptscriptfont\mbff=\mbss\def\mbf{\fam\mbff}

\def\Z{{\mbf Z}}
\def\Co{{\mbf C}}

\def\Di{{\mbf D}}
\def\N{{\mbf N}}

\newtheorem{Th}{Theorem}[section]
\newtheorem{Lm}[Th]{Lemma}
\newtheorem{C}[Th]{Corollary}
\newtheorem{D}[Th]{Definition}
\newtheorem{Proposition}[Th]{Proposition}
\newtheorem{R}[Th]{Remark}
\newtheorem{E}[Th]{Example}
\begin{document}

\title[Banach-valued Holomorphic Functions]{Banach-valued Holomorphic Functions on the Maximal Ideal Space of $H^\infty$}

\author{Alexander Brudnyi} 
\address{Department of Mathematics and Statistics\newline
\hspace*{1em} University of Calgary\newline
\hspace*{1em} Calgary, Alberta\newline
\hspace*{1em} T2N 1N4}
\email{albru@math.ucalgary.ca}
\keywords{Bounded holomorphic function, maximal ideal space, $\bar\partial$-equation, cohomology, slice algebra}
\subjclass[2000]{Primary 30D55. Secondary 30H05}

\thanks{Research supported in part by NSERC}

\begin{abstract}
We study Banach-valued holomorphic functions defined on open subsets of the maximal ideal space of the Banach algebra $H^\infty$ of bounded holomorphic functions on the unit disk $\Di\subset\Co$ with pointwise multiplication and supremum norm. In particular, we establish vanishing cohomology for sheaves of germs of such functions and, solving a Banach-valued corona problem for $H^\infty$, prove that the maximal ideal space of the algebra $H_{\rm comp}^\infty (A)$ of holomorphic functions on $\Di$ with relatively compact images in a commutative unital complex Banach algebra $A$ is homeomorphic to the direct product of maximal ideal spaces of $H^\infty$ and $A$. 
\end{abstract}

\date{}



\maketitle

\sect{Formulation of Main Results}
\subsect{}
The paper deals with Banach-valued holomorphic functions defined on open subsets of the maximal ideal space of the Banach algebra $H^\infty$ of bounded holomorphic functions on the unit disk $\Di\subset\Co$ with pointwise multiplication and supremum norm. As in the theory of Stein manifolds, we establish vanishing cohomology for sheaves of germs of such functions, solve the second Cousin problem and prove Runge-type approximation theorems for them. We then apply the developed technique to the study of algebra $H_{\rm comp}^\infty (A)$ of holomorphic functions on $\Di$ with relatively compact images in a commutative unital complex Banach algebra $A$. This class of algebras includes, e.g.,  slice algebras $S(H^\infty; \cdot)$. In particular, solving a Banach-valued corona problem for $H^\infty$, we prove that the maximal ideal space of $H_{\rm comp}^\infty(A)$ is homeomorphic to the direct product of maximal ideal spaces of $H^\infty$ and $A$. Recall that for a commutative unital complex Banach algebra $A$ with dual space $A^*$ the {\em maximal ideal space} $M(A)$ of $A$ is the set of nonzero homomorphisms $A\to\Co$ equipped with the {\em Gelfand topology}, the weak$^*$ topology induced by $A^*$. It is a compact Hausdorff space contained in the unit ball of $A^*$.
Let $C(M(A))$ be the algebra of
continuous complex-valued functions on $M(A)$ with supremum norm. The Gelfand transform $\hat \,: A\to C(M(A))$, defined by $\hat a(\varphi):=\varphi(a)$, is a nonincreasing-norm morphism of algebras that allows to thought of elements of $A$ as continuous functions on $M(A)$. 
If the Gelfand transform is an isometry (as for $H^\infty$), then $A$ is called a {\em uniform algebra}. 

Throughout the paper all Banach algebras are assumed to be complex, commutative and unital.

In the case of $H^\infty$ evaluation at a point of $\Di$ is an element of $M(H^\infty)$, so $\Di$ is naturally embedded into $M(H^\infty)$ as an open subset. The famous Carleson corona theorem \cite{C1} asserts that $\Di$ is dense in $M(H^\infty)$.

Next, given Banach algebras $A\subset C(X)$, $B\subset C(Y)$ ($X$ and $Y$ are compact Hausdorff spaces) their {\em slice algebra} is defined as
\[
S(A;B) := \{f \in C(X \times Y)\, ;\,  f(\cdot; y) \in A\ \text{for all}\ y \in Y\, ;\, f(x; \cdot)\in B\ \text{for all}\ x\in X\}.
\]
The main problem concerning algebra $S(A;B)$ is to determine whether it coincides with $A\otimes B$, the closure in $C(X\times Y )$ of the symmetric tensor product of $A$ and
$B$. For instance, this is true if either $A$ or $B$ have the {\em approximation property}. The latter
is an immediate consequence of the following result of Grothendieck \cite[Sect.~5.1]{G}.

Let $A\subset C(X)$ be a closed subspace, $B$ be a complex Banach space and $A_B \subset C(X;B)$
be the Banach space of all continuous $B$-valued functions $f$ such that $\varphi(f)\in A$ for any
$\varphi\in B^*$. By $A\otimes B$ we denote completion of symmetric tensor product of $A$ and $B$ with
respect to norm
\begin{equation}\label{eq1}
\left\|\sum_{k=1}^m a_k\otimes b_k\right\|:=\sup_{x\in X}\left\|\sum_{k=1}^m a_k(x) b_k\right\|_B\quad \text{with}\quad a_k\in A,\, b_k\in B.
\end{equation}
\begin{Th}[Grothendieck]\label{th1}
The following statements are equivalent:
\begin{itemize}
\item[(1)]
$A$ has the approximation property;
\item[(2)] 
$A\otimes B = A_B$ for every Banach space $B$.
\end{itemize}
\end{Th}

Recall that a Banach space $B$ is said to have the approximation property, if,
for every compact set $K\subset B$ and every $\varepsilon > 0$, there exists an operator $T : B\to B$ of
finite rank so that $\|Tx-x\|_B\le\varepsilon$ for every $x\in K$. 

Although it is strongly believed that the class of spaces with the approximation property
includes practically all spaces which appear naturally in analysis, it is not known yet even
for the space $H^\infty$. (The strongest result in this direction due to Bourgain and Reinov \cite[Th.~9]{BG} states that $H^\infty$ has the approximation property ``up to logarithm``.) The first example of a space which fails to have the approximation property was
constructed by Enflo \cite{E}. Since Enflo's work several other examples of such spaces were
constructed, for the references see, e.g., \cite{L}. 

In the paper we show that $H^\infty$ has the approximation property if and only if it has this property in some open neighbourhoods of trivial Gleason parts of $M(H^\infty)$.  
\subsection{}
Let $U\subset M(H^\infty)$ be an open subset and $B$ be a complex Banach space. 
\begin{D}\label{def1}
A continuous function $f\in C(U;B)$ is said to be $B$-valued holomorphic if its restriction to $U\cap\Di$ is $B$-valued holomorphic in the usual sense. 

By $\mathcal O(U;B)$ we denote the vector space of $B$-valued holomorphic functions on $U$. 
\end{D}
If $f\in\mathcal O(U; B)$, then for every open $V\Subset U$ the restriction $f|_{V\cap\Di}$ belongs to the Banach space 
$H_{\rm comp}^\infty(V\cap\Di; B)$ of $B$-valued holomorphic functions $g$ on $V\cap\Di$ with relatively compact images and with norm $\|g\|:=\sup_{z\in V\cap\Di}\|g(z)\|_B$.
Conversely, we have
\begin{Proposition}\label{sheaf}
Let $f\in \mathcal O(U\cap\Di;B)$ satisfy $f|_{V\cap\Di}\in H_{\rm comp}^\infty(V\cap\Di;B)$ for each open $V\Subset U$. Then there exists a unique function $\tilde f\in\mathcal O(U;B)$ such that
$\tilde f|_{U\cap\Di}=f$.
\end{Proposition}

By $\mathcal O_{M(H^\infty)}^B$ we denote the sheaf of germs of $B$-valued holomorphic functions on $M(H^\infty)$. 
\begin{Th}\label{main}
$H^k(M(H^\infty);\mathcal O_{M(H^\infty)}^B)=0$.
\end{Th}
\noindent Here $H^k(X;\mathcal J)$ stands for the $k\,$th {\em \v{C}ech cohomology group} of a sheaf of abelian groups $\mathcal J$ defined on a Hausdorff topological space $X$.

The proof of Theorem \ref{main} is based on a new method for solving of certain Banach-valued $\bar\partial$-equations on $\Di$, see Theorem \ref{dbar} (this theorem is the main tool in most of our proofs).

A function $h$ on an open subset $U\subset M(H^\infty)$ is said to be {\em meromorphic} if $h=\frac fg$, where $f,g\in\mathcal O(U)\, (:=\mathcal O(U;\Co))$ and $g$ is not identically zero. The set of meromorphic functions on $U$ is denoted by $\mathcal M(U)$. For $U=M(H^\infty)$ the class $\mathcal M(U)$ consists of functions $h=\frac fg$ with $f, g\in H^\infty$. (Here and below $H^\infty$ is defined on $M(H^\infty)$ by means of the Gelfand transform.) 

A (Cartier) {\em divisor} on $M(H^\infty)$ consists of pairs $(U_i, h_i)$, where $(U_i)$ is an open cover of $M(H^\infty)$ and $h_i\in\mathcal M(U_i)$, such that for all $U_i\cap U_j\ne\emptyset$ functions
$\frac{h_i}{h_j}\in\mathcal O(U_i\cap U_j)$ and are nowhere zero. As a corollary of Theorem \ref{main} we obtain the solution of the second Cousin problem on $M(H^\infty)$.
\begin{Th}\label{cousin}
For any divisor $D=\{(U_i,h_i)\}_{i\in I}$ on $M(H^\infty)$ there exist a meromorphic function $h_D\in\mathcal M(M(H^\infty))$ and a family of nowhere vanishing functions $c_i\in \mathcal O(U_i)$, $i\in I$, such that
\[
h_D|_{U_i}=h_i\cdot c_i\quad\text{for all}\quad i\in I.
\]
\end{Th}
\begin{R}{\rm The statement is equivalent to the fact that any holomorphic line bundle on $M(H^\infty)$ (i.e., a line bundle determined on an open cover of $M(H^\infty)$ by a holomorphic cocycle) is holomorphically trivial. The general question of whether any finite rank holomorphic bundle on $M(H^\infty)$ is holomorphically trivial is open.}
\end{R}
Our next result is a Runge-type approximation theorem for Banach-valued holomorphic functions defined on subsets of $M(H^\infty)$.

A compact subset $K\subset M(H^\infty)$ is called {\em holomorphically convex} if for any $x\notin K$ there is $f\in H^\infty$ such that
$\max_{K}|f|<|f(x)|$.
\begin{Th}\label{runge}
Any $B$-valued holomorphic function defined on a neigbourhood of a holomorphically compact set $K\subset M(H^\infty)$ can be uniformly approximated on $K$ by functions from $\mathcal O(M(H^\infty);B)$. 
\end{Th}

In fact it suffices to prove the result for $K$ a {\em polyhedron}, i.e., a set of the form $\Pi(F_\ell):=\{x\in M(H^\infty)\, ;\, \max_{1\le j\le\ell}|f_j(x)|\le 1,\, F_\ell:=(f_1,\dots, f_\ell)\in (H^\infty)^{\ell}\,\}$ (see Lemma \ref{polyhedr}). Let $\dot \Pi(F_\ell)$ denote intersection of the interior of $\Pi(F_\ell)$ with $\Di$.
Then in view of Proposition \ref{sheaf}, we can reformulate Theorem \ref{runge} as follows.

{\em Any function in $H_{\rm comp}^\infty(\dot\Pi(F_\ell);B)$ 
can be uniformly approximated on each $\Pi(r\cdot F_\ell)$, $r>1$, by functions from $H_{\rm comp}^\infty(B):=H_{\rm comp}^\infty(\Di;B)$.}
\begin{R}\label{merg}
{\rm With regard to Theorem \ref{runge} one may ask about an analog of the Mergelyan theorem. For a polyhedron the question can be stated as follows.}

Is it true that any function in $C(\Pi(F_\ell);B)$ holomorphic on $\dot\Pi(F_\ell)$ can be uniformly approximated on $\Pi(F_\ell)\cap\Di$ by functions from $H_{\rm comp}^\infty(B)$?
\end{R}
\subsection{}
In this part we study Banach algebras $H_{\rm comp}^\infty(A)$ of holomorphic functions on $\Di$ with relatively compact images in a  commutative complex unital Banach algebra $A$. This class of algebras includes, e.g., slice algebras $S(H^\infty;\cdot)$. Note that the restriction map to $\Di$ induces an isomorphism between $\mathcal O(M(H^\infty) ; A)$ and $H_{\rm comp}^\infty(A)$, cf. Proposition \ref{sheaf}. 
\begin{Th}\label{th2}
Let $f_1,\dots, f_m,\, f\in \mathcal O(M(H^\infty);A)$. Then $f$ belongs to the ideal $I\subset  \mathcal O(M(H^\infty);A)$ generated by $f_1,\dots, f_m$ if and only if there exists a finite open cover
$(U_k)_{1\le k\le\ell}$ of $M(H^\infty)$ such that for every $1\le k\le\ell$ the function $f|_{U_k}$ belongs to the ideal $I_k\subset \mathcal O(U_k;A)$ generated by functions $f_1|_{U_k},\dots , f_m|_{U_k}$.
\end{Th}

In the proof we use a standard argument involving Koszul complexes which reduces the statement to a question on existence of bounded on the boundary solutions of certain  $A$-valued $\bar\partial$-equations on $\Di$ similar to those in Wolff's proof of Carleson's corona theorem, see, e.g., \cite[Ch.~VIII.2]{Ga}. However, since the target space $A$ may be infinite dimensional, the classical duality method allowing to get such solutions for scalar  $\bar\partial$-equations does not work anymore. We use instead some topological properties of $M(H^\infty)$ together with our Theorem \ref{dbar} which provides bounded solutions of $A$-valued $\bar\partial$-equations with `supports` in the set of nontrivial Gleason parts of $M(H^\infty)$. The fact that $\Di$ is dense in $M(H^\infty)$ is not used in the proof; hence, Theorem \ref{th2} gives yet another proof of the corona theorem for $H^\infty$.

As a corollary we obtain
\begin{Th}\label{th3}
$M(H_{\rm comp}^\infty(A))\cong M(H^\infty)\times M(A)$.
\end{Th}
This would follow from Theorem \ref{th1} if we knew that $H^\infty$ has the approximation property. However, Theorem \ref{th2} yields also results not following from this property; one of them, an analog of Wolff's theorem, see, e.g., \cite[Ch.~VIII, Th.~2.3]{Ga}, is presented below.

Recall that $M(H^\infty)$ is disjoint union of an open subset $M_a$ of nontrivial Gleason parts ({\em analytic disks}) and a closed subset $M_s$ of trivial (one-pointed) Gleason parts (see Section 2 for definitions).
\begin{Th}\label{square}
Let the algebra $A\subset C(X)$, $X$ a compact Hausdorff space, be self-adjoint with respect to the complex conjugation. Fix an $\omega\in C\bigl([0,\infty)\bigr)$ positive on $(0,\infty)$ such that $\lim_{t\to 0^+}\frac{\omega(t)}{t^3}=0$.
Assume that $f_1,\dots, f_m,\, f\in S(A):=S(H^\infty ;A)$ satisfy 
\begin{itemize}
\item[(A)]
\[
|f(z)|\le c\cdot\omega\left(\max_{1\le j\le m}|f_j(z)|\right)\quad\text{for some}\quad c>0\quad\text{and all}\quad z\in\Di\times X ;
\]
\item[(B)]
\[
\max_{1\le j\le m}|f(x)|\ge\delta\quad\text{for some}\quad \delta>0\quad\text{and all}\quad x\in M_s\times X .
\]
\end{itemize} 
Then $f$ belongs to the ideal $I\subset S(A)$ generated by $f_1,\dots, f_m$.
\end{Th}
\begin{E}\label{ex1}
{\rm A sequence $\{z_n\}\subset\Di$ is called {\em interpolating} for $H^\infty$ if the interpolation problem $f(z_n)=w_n$, $n\in\N$, has a solution $f\in H^\infty$ for every bounded sequence of complex numbers $\{w_n\}$. A {\em Blaschke product} $B(z):=\prod_{j\ge 1}\frac{z-z_j}{1-\bar z_j z}$, $z\in\Di$, is called interpolating if its set of zeros $\{z_j\}$ is an interpolating sequence for $H^\infty$.

Consider interpolating Blaschke products $B_1,\dots, B_m$. Then each $|B_j|$ is strictly positive on $M_s$, see, e.g., \cite[Ch.~X]{Ga}. 
Assume that $f_1,\dots, f_m,\, f\in S(A)$, where $A\subset C(X)$ is self-adjoint, satisfy for
some $c, \delta>0$
\begin{equation}\label{eq1.2}
\max_{1\le j\le m}|f_j(w)|\ge\delta\quad\text{for all}\quad w\in M(S(A));
\end{equation}
\begin{equation}\label{eq1.3}
|f(z,x)|\le c\cdot\omega\left(\max_{1\le j\le m}|B_j(z)f_j(z,x)|\right)\quad \text{for all}\quad (z,x)\in\Di\times X.
\end{equation}
Then $f$ belongs to the ideal $I\subset S(A)$ generated by $B_1f_1,\dots, B_mf_m$.

In fact, the pullbacks of $B_1,\dots, B_m$ to $M(H^\infty)\times X$ satisfy  condition (B) of Theorem \ref{square}.
This and \eqref{eq1.2} imply that $B_1f_1,\dots, B_mf_m$ satisfy this condition as well. Thus the required result follows from the theorem.
}
\end{E}
\begin{R}\label{product}
{\rm Assumptions of Theorem \ref{square} do not imply that $f\in I$. Indeed, if $B_1$, $B_2$ are interpolating Blaschke products such
that $\inf_{z\in\Di} (|B_1(z)|+|B_2(z)|)=0$, then $|B_1(z)B_2(z)|\le\max\{|B_1^2(z)|,|B_2^2(z)|\}$, $z\in\Di$, and $\max\{|B_1^2|,|B_2^2|\}$ is strictly positive on $M_s$, but $B_1B_2$ does not belong to the ideal $I\subset H^\infty$ generated by $B_1^2, B_2^2$, see \cite{R}. It is still unclear whether the conclusion of the theorem is valid under assumption (A) only, and whether the condition $\lim_{t\to 0^+}\frac{\omega(t)}{t^3}=0$ can be replaced by $\lim_{t\to 0^+}\frac{\omega(t)}{t^2}=0$ (cf. \cite[Ch.~VIII.2]{Ga} for a similar problem).
}
\end{R}

Let $S_N(H^\infty):=S(H^\infty;\dots ; H^\infty)$ be the $N$-dimensional slice algebra on $M(H^\infty)^N$ of continuous functions $f$ such that
$f(x, \cdot, y)\in H^\infty$ for each  $x\in M(H^\infty)^{k-1}$ and $y\in M(H^\infty)^{N-k}$, $k=1,\dots , N$. A major open problem posed in the mid of 1960s asks whether the maximal ideal space of $S_N(H^\infty)$ is $M(H^\infty)^N$, see, e.g., \cite{Cu} and references therein. This fact is obtained now as a corollary of Theorem \ref{th3}.
\begin{C}\label{cor1}
$M(S_N(H^\infty))=M(H^\infty)^N$.
\end{C} 

\subsection{}
In this subsection we apply Theorem \ref{th3} to the study of certain operator corona problems, analogs of the Sz.-Nagy problem \cite{SN} on existence of a bounded holomorphic left inverse to an $H^\infty$ function on $\Di$ with values in the space of bounded linear operators between two separable Hilbert spaces. 

To formulate the result we recall the definition of covering dimension:

For a normal space $X$, $\text{dim}\, X\le n$ if every finite open cover
of $X$ can be refined by an open cover whose order $\le n + 1$. If $\text{dim}\, X\le n$ and
the statement $\text{dim}\, X\le n-1$ is false, we say $\text{dim}\, X = n$. For the empty set,
$\text{dim}\, \emptyset = -1$.

For a commutative complex unital Banach algebra $A$ by $M^{k,n}(A)$ we denote the space of $k\times n$ matrices, $k\le n$, with entries in $A$.
\begin{Th}\label{th4}
Let $F=(f_{ij}):\Di\to M^{k,n}(A)$, $k<n$, be such that each $f_{ij}\in H_{\rm comp}^\infty(A)$. Assume that there exists $\delta>0$ such that the family $h_1,\dots, h_\ell\in H^\infty(A)$ of minors of $F$ of order $k$ satisfies the corona condition:
\begin{equation}\label{eq2}
\sum_{i=1}^\ell |\varphi(h_i(z))|\ge\delta\quad\text{for all}\quad z\in\Di\quad\text{and}\quad \varphi\in M(A).
\end{equation}

If $M(A)$ is the inverse limit of a family of compact Hausdorff spaces $\{M_\alpha\}_{\alpha\in\Lambda}$ such that each $M_\alpha$ is homotopically equivalent to a metrizable compact space $X_\alpha$ with ${\rm dim}\,X_\alpha\le d$ and if $n-k-1\ge \lfloor\frac d2\rfloor$ for $d\geq 3$, and $n-k-1\ge 0$ otherwise, then there exists a matrix-function
$\widetilde F=(\widetilde f_{ij}):\Di\to M^{n,n}(A)$ with entries in $H_{\rm comp}^\infty(A)$ such that ${\rm det}\,\widetilde F=1$ and
$\widetilde f_{ij}=f_{ij}$ for $1\le i\le k$, $1\le j\le n$.
\end{Th}
\begin{R}\label{rem1}
{\rm
(a) According to the result of Marde\v{s}i\'c \cite{M} any compact Hausdorff space $X$ with ${\rm dim}\, X\le d<\infty$ can be presented as the inverse limit of a family of metrizable compact spaces $X_\alpha$ with ${\rm dim}\, X_\alpha\le d$. In particular, Theorem \ref{th4} holds under the assumptions that ${\rm dim}\, M(A)\le d$ and $n-k-1\ge \lfloor\frac d2\rfloor$ for $d\ge 3$, and $n-k-1\ge 0$ otherwise.\\
(b) Theorem \ref{th4} provides conditions for an extension of $F$ up to an invertible bounded holomorphic $M^{n,n}(A)$-valued function. 
This, in particular, implies that under the conditions of the theorem $F$ is left invertible in the considered class.}
\end{R}

A Banach algebra $H_{\rm comp}^\infty(A)\, (\cong \mathcal O(M(H^\infty);A))$ for which Theorem \ref{th4} is valid for all $1\le k<n<\infty$ is called {\em Hermite}. (An equivalent definition is
that every finitely generated stably free $H_{\rm comp}^\infty(A)$-module is free.)

A Banach algebra $H_{\rm comp}^\infty(A)$ is called {\em projective free} if for any $n\in\N$ every idempotent matrix $F:\Di\to M^{n,n}(A)$ (i.e., such that $F^2=F$) with entries in $H_{\rm comp}^\infty(A)$ is conjugate by means of an invertible matrix $H:\Di\to M^{n,n}(A)$ with entries in $H_{\rm comp}^\infty(A)$ to a constant (idempotent) matrix. (An equivalent definition is that every finitely generated projective $H_{\rm comp}^\infty(A)$-module is free.) 

Note that any projective free algebra is Hermite.

\begin{Th}\label{th5}
If the maximal ideal space $M(A)$ of $A$  is the inverse limit of a family of metrizable contractible compact spaces, then the algebra
$H_{\rm comp}^\infty(A)$ is projective free.
\end{Th}

\begin{E}\label{ex2}
{\rm
(1) In \cite{S1} Su\'{a}rez proved that ${\rm dim}\, M(H^\infty)=2$. Therefore by Corollary \ref{cor1}, ${\rm dim}\, S_N(H^\infty)={\rm dim}\ M(H^\infty)^N=2N$. It is known that $H^\infty$ is projective free, see, e.g., \cite{Q}. Applying Theorem \ref{th4} we obtain
(a)
$S_2(H^\infty)$ is Hermite; (b)
If $N\ge 3$, then for a matrix $F=(f_{ij}):M(S_N(H^\infty))\to M^{k,n}(\Co)$, $n-k\ge N$, with entries in $S_N(H^\infty)$ whose family of minors $h_1,\dots, h_\ell$ of order $k$ satisfies 
\[
\sum_{i=1}^\ell |h_i(z)|>0\quad\text{for all}\quad z\in M(S_N(H^\infty)),
\]
there exists a matrix 
$\widetilde F=(\widetilde f_{ij}):\Di\to M^{n,n}(\Co)$ with entries in $S_N(H^\infty)$ such that ${\rm det}\,\widetilde F=1$ and
$\widetilde f_{ij}=f_{ij}$ for $1\le i\le k$, $1\le j\le n$.

\noindent (2) Let $G$ be a compact connected abelian topological group. A discrete (additive) semigroup $\Sigma_*\subset\widehat G$ ($:=$ the Pontriagin dual group of $G$) is called {\em pointed} if $0\in\Sigma_*$ and if $\pm x\in\Sigma_*$ implies that $x=0$. (It is known that one can introduce an order $\ge 0$ on $\widehat G$; then, e.g., $\Sigma_+:=\{x\in\widehat G\, ;\, x\ge 0\}$ and $\Sigma_-:=\{x\in\widehat G\, ;\, -x\ge 0\}$ are pointed semigroups.) 

Let $W_{\Sigma_*}(G)\subset C(G)$ be the Wiener algebra of functions with Bohr-Fourier spectra in $\Sigma_*$. It was shown in \cite{BRS} that the maximal ideal space of $W_{\Sigma_*}(G)$ is the inverse limit of a family of metrizable contractible compact spaces. Hence, Theorem \ref{th5} implies that $H_{\rm comp}^\infty(W_{\Sigma_*}(G))$ is projective free. In particular, $H_{\rm comp}^\infty(W_{\Sigma_*}(G))$ is Hermite.}
\end{E}

\subsection{}
Finally, we reformulate the problem on the approximation property for $H^\infty$ in local terms. 

For an open subset $U\subset M(H^\infty)$ by $H^\infty(U)$ we denote the Banach space of bounded holomorphic functions on $U$
with supremum norm. According to Proposition \ref{sheaf},
$H^\infty(U)$ is isomorphic to $H^\infty(U\cap \Di)$. 
\begin{Th}\label{approx}
Let $U\Subset V\subset M_a$ be open subsets and $B$ be a complex Banach space. Then for any function $f\in\mathcal O(V;B)$ its restriction
$f|_{U}$ belongs to $H^\infty(U)\otimes B$. 

In particular,
$H^\infty$ has the approximation property iff
for every complex Banach space $B$ and any function $f\in {\mathcal O}(M(H^\infty);B)$ there exists an open cover $(U_i)_{i\in I}$ of the set $M_s$ of trivial Gleason parts such that $f|_{U_i}\in H^\infty(U_i)\otimes B$ for all $i\in I$.
\end{Th}

In Section~9 we extend part of our results to holomorphic functions defined on certain Riemann surfaces, e.g., on Riemann surfaces of finite type.

\sect{Maximal Ideal Space of $H^\infty$}
In this section we collect some auxiliary results on the structure of the maximal ideal space of $H^\infty$.

\subsection{}
Recall that the pseudohyperbolic metric on $\Di$ is defined by
\[
\rho(z,w):=\left|\frac{z-w}{1-\bar w z}\right|,\qquad z,w\in\Di.
\]
For $x,y\in\mathcal M(H^\infty)$ the formula
\[
\rho(x,y):=\sup\{|\hat f(y)|\, ;\, f\in H^\infty,\, \hat f(x)=0,\, \|f\|\le 1\}
\]
gives an extension of $\rho$ to $M(H^\infty)$.
The {\em Gleason part} of $x\in\mathcal M$ is then defined by $\pi(x):=\{y\in M(H^\infty)\, ;\, \rho(x,y)<1\}$. For $x,y\in M(H^\infty)$ we have
$\pi(x)=\pi(y)$ or $\pi(x)\cap\pi(y)=\emptyset$. Hoffman's classification of Gleason parts \cite{H} shows that there are only two cases: either $\pi(x)=\{x\}$ or $\pi(x)$ is an analytic disk. The former case means that there is a continuous one-to-one and onto map $L_x:\Di\to\pi(x)$ such that
$\hat f\circ L_x\in H^\infty$ for every $f\in H^\infty$. Moreover, any analytic disk is contained in a Gleason part and any maximal (i.e., not contained in any other) analytic disk is a Gleason part. By $M_a$ and $M_s$ we denote the sets of all non-trivial (analytic disks) and trivial (one-pointed) Gleason parts, respectively. It is known that $M_a\subset M(H^\infty)$ is open. Hoffman proved that $\pi(x)\subset M_a$ if and only if $x$ belongs to the closure of some interpolating sequence in $\Di$. 

\subsection{Structure of $M_a$}
In \cite{Br1} $M_a$ is described as a fibre bundle over a compact Riemann surface. Specifically, let $G$ be the fundamental group of a compact Riemann surface $S$ of genus $\ge 2$. Let $\ell_\infty(G)$ be the Banach algebra of bounded complex-valued functions on $G$ with pointwise multiplication and supremum norm. By $\beta G$ we denote the {\em Stone-\v{C}ech compactification} of $G$, i.e., the maximal ideal space of $\ell_\infty(G)$ equipped with the Gelfand topology.
 
The universal covering $r:\Di\to S$ is a principal fibre bundle with fibre $G$. Namely, there exists a finite open cover ${\mathcal U} = (U_i)_{i\in I}$ of $S$ by sets biholomorphic to $\Di$ and a locally constant cocycle $\bar g=\{g_{ij}\}\in Z^1({\mathcal U}; G)$ such that $\Di$ is biholomorphic to the quotient space of the disjoint union $V=\sqcup_{i\in I}U_i\times G$ by the equivalence relation $U_i\times G\ni (x, g)\sim (x, gg_{ij})\in U_j\times G$. The identification space is a fibre bundle with projection $r:\Di\to S$ induced by projections $U_i\times G\to U_i$, see, e.g., \cite[Ch.~1]{Hi}.

Next, the right action of $G$ on itself by multiplications is extended to the right continuous action of $G$ on $\beta G$. Let $\tilde r: E(S,\beta G)\to S$ be the associated with this action bundle on $S$ with fibre $\beta G$ constructed by cocycle $\bar g$. Then $E(S,\beta G)$ is a {\em compact} Hausdorff space homeomorphic to the quotient space of the disjoint union $\widetilde V=\sqcup_{i\in I}U_i\times \beta G$ by the equivalence relation $U_i\times \beta G\ni (x, \xi)\sim (x, \xi g_{ij})\in U_j\times \beta G$. The projection $\tilde r:E(S,\beta G)\to S$ is induced by projections $U_i\times \beta G\to U_i$.
Note that there is a natural embedding $V\hookrightarrow\widetilde V$ induced by the embedding $G\hookrightarrow\beta G$. This embedding commutes
with the corresponding equivalence relations and so determines an embedding of $\Di$ into $E(S,\beta G)$ as an open dense subset.
Similarly, for each $\xi\in\beta G$ there exists a continuous injection  $V\to\widetilde V$ induced by the injection $G\to\beta G$, $g\mapsto\xi g$,
commuting with the corresponding equivalence relations. Thus it determines a continuous injective map $i_{\xi}:\Di\to E(S,\beta G)$.
Let $X_G:=\beta G/G$ be the set of co-sets with respect to the right action of $G$ on $\beta G$. Then $i_{\xi_1}(\Di)=i_{\xi_2}(\Di)$ if and only if $\xi_1$ and $\xi_2$ determine the same element of $X_G$.  If $\xi$ represents an element $x\in X_G$, then we write $i_x(\Di)$ instead of $i_{\xi}(\Di)$. In particular, $E(S,\beta G)=\sqcup_{x\in X_G}i_x(\Di)$.

Let $U\subset E(S,\beta G)$ be open. We say that a function $f\in C(U)$ is holomorphic if $f|_{U\cap\Di}$ is holomorphic in the usual sense. The set of holomorphic on $U$ functions is denoted by $\mathcal O(U)$. It was shown in \cite[Th.~2.1]{Br1} that each $h\in H^\infty(U\cap\Di)$ is extended to a unique holomorphic function $\hat h\in \mathcal O(U)$. In particular, the restriction map $\mathcal O(E(S,\beta G))\to H^\infty(\Di)$ is an isometry of Banach algebras. Thus the quotient space of $E(S,\beta G)$ (equipped with the factor topology) by the equivalence relation
$x\sim y \Leftrightarrow f(x)=f(y)$ for all $f\in \mathcal O(E(S,\beta G))$ is homeomorphic to $M(H^\infty)$. By $q$ we denote the quotient map 
$E(S,\beta G)\to M(H^\infty)$.

A sequence $\{g_n\}\subset G$ is said to be interpolating if $\{g_n(0)\}\subset\Di$ is interpolating for $H^\infty$ (here $G$ acts on $\Di$ by M\"{o}bius transformations). Let $G_{in}\subset\beta G$ be the union of closures of all interpolating sequences in $G$. It was shown that $G_{in}$ is an open dense subset of $\beta G$ invariant with respect to the right action of $G$.
The associated with this action bundle $E(S, G_{in})$ on $S$ with fibre $G_{in}$ constructed by the cocycle $\bar g\in Z^1(\mathcal U ;G)$ is an open dense subbundle of $E(S,\beta G)$ containing $\Di$. It was established in \cite{Br1} that 
$q$ maps $E(S, G_{in})$ homeomorphically onto $M_a$ so that for each $\xi\in G_{in}$ the set $q\bigl(i_{\xi}(\Di)\bigr)$ coincides with the Gleason part $\pi\bigl(q(i_\xi (0))\bigr)$.  Also, for distinct $x,y\in E(S,\beta G)$ with $x\in E(S, G_{in})$ there exists $f\in  \mathcal O(E(S,\beta G))$ such that $f(x)\ne f(y)$. Thus $q(x)=x$ for all $x\in E(S,G_{in})$, i.e.,  $E(S,G_{in})=M_a$.

It is worth noting that every bounded uniformly continuous (Lipschitz) with respect to the metric $\rho$ function $f$ on $\Di$ admits a continuous extension $\hat f$ to $E(S,\beta G)$ (and, in particular, to $M_a$) so that for every map $i_{\xi}:\Di\to E(S,\beta G)$ the function $\hat f\circ i_{\xi}$ is uniformly continuous (Lipschitz) with respect to $\rho$ on $\Di$.

From the definition of $E(S, \beta G)$ follows that for a simply connected open subset $U\subset S$ restriction $E(S, \beta G)|_U\, \bigl(:=\tilde r^{-1}(U)\bigr)$ is a trivial bundle, i.e., there exists an isomorphism of bundles (with fibre $\beta G$) 
$\varphi: E(S, \beta G)|_U\to U\times \beta G$, $\varphi(x):=(\tilde r(x),\tilde\varphi(x))$, $x\in E(S,\beta G)|_U$, mapping $\tilde r^{-1}(U)\cap\Di$ biholomorphically onto $U\times G$.
A subset $W\subset\tilde r^{-1}(U)$ of the form $R_{U,H}:=\varphi^{-1}(U\times H)$, $H\subset \beta G$, is called {\em rectangular}. The base of topology on $E(S, G_{in})\, (:=M_a)$ consists of rectangular sets $R_{U,H}$ with $U\subset S$ biholomorphic to $\Di$ and $H\subset G_{in}$ being the closure of an interpolating sequence in $G$ (so $H$ is a clopen subset of $\beta G$). 
Another base of topology on $M_a$ is given by sets of the form $\{x\in M_a\, ;\, |\hat B(x)|<\varepsilon\}$, where $B$ is an interpolating Blaschke product. This follows from the fact that for a sufficiently small $\varepsilon$ the set $B^{-1}(\Di_\varepsilon)\subset\Di$, $\Di_\varepsilon:=\{z\, ;\, |z|<\varepsilon\}$, is biholomorphic to $\Di_\varepsilon\times B^{-1}(0)$, see \cite[Ch.~X, Lm.~1.4]{Ga}. Hence,
$\{x\in M_a\, ;\, |\hat B(x)|<\varepsilon\}$ is biholomorphic to $\Di_\varepsilon\times \hat B^{-1}(0)$.
\subsection{Structure of $M_s$} It was proved in \cite{S2}, that the set $M_s$ of trivial Gleason parts is totally disconnected, i.e., ${\rm dim}\, M_s=0$ (because $M_s$ is compact). Moreover, $C(M_s)$ is the uniform closure of the algebra 
\[
A(M_s):=\left\{\frac{\hat f|_{M_s}}{\hat g|_{M_s}}\, ;\, f, g\in H^\infty\ \text{and}\ \hat g\ \text{never vanishes on}\ M_s\right\}.
\] 


%
\sect{$\bar\partial$-equations with Support in $M_a$}
\subsection{}
First, we develop differential calculus on $E(S,\beta G)$.

Let $X$ be a complex Banach space and $U\subset S$ be open simply connected. Let $\varphi: E(S,\beta G)|_U\to U\times \beta G$ be a trivialization as in section 2.2. We say that a function $f\in C(E(S,\beta G)|_U;X)$ belongs to the space $C^k(E(S,\beta G)|_U;X)$, $k\in\N\cup\{\infty\}$, if its pullback to $U\times \beta G$ by $\varphi^{-1}$ is of class $C^k$. In turn, a continuous $X$-valued function on $U\times \beta G$ is of class $C^k$ if regarded as a Banach-valued map $U\to C(\beta G;X)$ it has continuous derivatives of order $\le k$ (in local coordinates on $U$).

\begin{Lm}\label{le1}
The above definition does not depend on the choice of $\varphi$.
\end{Lm}
\begin{proof}
Suppose $\varphi': E(S, \beta G)|_U\to U\times \beta G$, $\varphi'(x):=(\tilde r(x),\tilde\varphi'(x))$, $x\in E(S,\beta G)|_U$, is another trivialization. Then $\psi:=\varphi'\circ\varphi^{-1}: U\times \beta G\to U\times \beta G$ is a homeomorphism of the form
$\psi(z,\xi)=(z,\tilde\psi(z,\xi))$, $(z,\xi)\in U\times \beta G$. Since $U$ is connected and $\beta G$ is totally disconnected, the continuous map $\tilde\psi:U\times \beta G\to \beta G$ does not depend on the first coordinate, i.e., $\tilde\psi(z,\xi):=\eta(\xi)$ for a homeomorphism $\eta: \beta G\to \beta G$ such that $\eta(G)=G$. To prove the lemma it suffices to show that if $f$ is of class $C^k$ on $U\times \beta G$, then $f\circ\psi$ is of class $C^k$ on $U\times \beta G$ as well. But $(f\circ\psi)(z,\xi):=f(z,\eta(\xi))$ and $\eta$ induces a linear isomorphism of the Banach space $C(\beta G;X)$. Therefore $f\circ\psi$ regarded as a map $U\to C(\beta G,X)$ is the composite of this isomorphism and
the map $f:U\to C(\beta; X)$. Hence, $f\circ\psi$ is of class $C^k$ on $U\times\beta G$.
\end{proof}

For a rectangular set $R_{U,H}\subset E(S,\beta G)|_U$ with clopen $H$ a function $f$ on $R_{U,H}$ is said to belong to the space $C^k(R_{U,H};X)$ if its extension to $E(S,\beta G)|_U$ by $0$ belongs to $C^k(E(S,\beta G)|_U;X)$. 

For an open $V\subset E(S;\beta G)$ a continuous function $f$ on $V$ belongs to the space $C^k(V;X)$ if its restriction to each
$R_{U,H}\subset V$  with $H$ clopen belongs to $C^k(R_{U,H};X)$. 

In the proofs we use the following result.
\begin{Proposition}\label{partition}
For a finite open cover of $E(S;\beta G)$ there exists a $C^\infty$ partition of unity subordinate to it.
\end{Proposition}
\begin{proof}
A rectangular set $R_{U,H}\subset E(S,\beta G)$ with $U$ biholomorphic to $\Di$ and $H$ clopen will be called a {\em coordinate chart}. 
\begin{Lm}\label{le4.1}
Let $U\subset E(S,\beta G)$ be open and $K\subset U$ be a compact subset. Then there exists a nonnegative $C^\infty$ function $\psi$ on $E(S,\beta G)$ such that $\psi|_K>0$ and ${\rm supp}\,\psi\subset U$.
\end{Lm}
\begin{proof}
We cover $K$ by two families of coordinate charts $(R_{U_i,H_i})_{1\le i\le k}$, $(R_{U_i',H_i})_{1\le i\le k}$ such that 
$U_i\Subset U_i'$ and $R_{U_i',H_i}\Subset U$ for all $i$. Consider the nonnegative $C^\infty$ function $\psi_i:=\rho_i\cdot\chi_i$ on $E(S,\beta G)|$, where $\rho_i$ is the pullback by $\tilde r$ of a nonnegative $C^\infty$ function on $S$ equals $1$ on $U_i$ and having support in $U_i'$ and $\chi_i$ is the characteristic function of $R_{U_i',H'}$. Then $\psi_i|_{R_{U_i,H_i}}=1$ and ${\rm supp}\,\psi_i\subset R_{U_i',H_i}$. The function
$\psi:=\sum_{i=1}^k\psi_i$ satisfies the required properties.
\end{proof}
Now, assume that ${\mathcal U}=(U_i)_{1\le i\le k}$ is a finite open cover of the compact Hausdorff space $E(S,\beta G)$. Then there exists a finite open refinement $(W_i)_{1\le i\le k}$ of ${\mathcal U}$ such that $\bar W_i\subset U_i$ for all $i$. By Lemma \ref{le4.1} there is a nonnegative $C^\infty$ function $\psi_i$ on $E(S,\beta G)$ such that $\psi_i|_{W_i}>0$ and ${\rm supp}\,\psi_i\subset U_i$. We set
\[
\varphi_i:=\frac{\psi_i}{\sum_{j=1}^k \psi_j}. 
\]
Then $\{\varphi_i\}_{1\le i\le k}$ is a $C^\infty$ partition of unity subordinate to $\mathcal U$.
\end{proof}
\begin{C}\label{cutoff1}
Let $U\Subset V\subset E(S,\beta G)$ be open. Then there exists a nonnegative $C^\infty$ function $\rho$ on $E(S,\beta G)$ such that $\rho=1$ in an open neighbourhood of $\bar U$ and ${\rm supp}\,\rho\subset V$.
\end{C}
\begin{proof}
Let $\{\varphi_i\}_{i=1,2}$ be a $C^\infty$ partition of unity subordinate to the cover $(U_i)_{i=1,2}$ of $E(S,\beta G)$, where $U_1:=V$, $U_2:=E(S,\beta G)\setminus\bar U$. Then $\rho:=\varphi_1$ is as required.
\end{proof}
An $X$-valued $(0,1)$-form $\omega$ of class $C^k$ on an open $U\subset E(S,\beta G)$ is defined in each coordinate chart $R_{V,H}\subset U$ with local coordinates $(z,\xi)$ (pulled back from $V\times\beta G$ by $\varphi$) by the formula $\omega|_{R_{V,H}}:=f(z,\xi)d\bar z$, $f\in C^k(R_{V, H};X)$, so that the restriction of the family $\{\omega|_{R_{V,H}}\,;\,R_{V,H}\subset U\}$ to $\Di$ determines a global $X$-valued $(0,1)$-form of class $C^k$ on the open set $U\cap\Di\subset\Di$. (If $U=E(S,\beta G)$, then such $\omega$ can be equivalently defined as a $C^k$ $(0,1)$-form on $S$ with values in the Banach holomorphic vector bundle $E_X$ on $S$ with fibre $C(\beta G;X)$ associated with the action of $G$ on $C(\beta G;X)$: $g\in G$ maps $h(x)\in C(\beta G;X)$ into $h(xg)$.) 

By $\mathcal E ^{0,1}(U; X)$ we denote the space of $X$-valued $(0,1)$-forms on $U\subset E(S,\beta G)$. The operator $\bar\partial: C^\infty(U;X)\to \mathcal E ^{0,1}(U; X)$ is defined in each $R_{V,H}\subset U$ equipped with the local coordinates $(z,\xi)$ as $\bar\partial f(z,\xi):=\frac{\partial f}{\partial\bar z}(z,\xi)d\bar z$. Then the composite of the restriction map to $U\cap\Di$ with this operator coincides with the standard $\bar\partial$ operator defined on $C^\infty(U\cap\Di;X)$.
(For $U=E(S,\beta G)$, identifying $C^\infty(E(S,\beta G);X)$ and $\mathcal E ^{0,1}(E(S,\beta G); X)$ with spaces of $C^\infty$ sections of the bundle $E_X$ and of $C^\infty$ $(0,1)$ forms with values in the fibres of this bundle, we obtain that $\bar\partial$ is the standard operator between these spaces.)

It is easy to check, using Cauchy estimates for derivatives of families of uniformly bounded holomorphic functions on $\Di$, that if $f\in {\mathcal O}(U; X)$, $U\subset E(S,\beta G)$, then $f\in C^\infty(U; X)$ and in each $R_{V,H}\subset U$ with local coordinates $(z,\xi)$ the function $f(z,\xi)$ is holomorphic in $z$. Thus $\bar\partial f=0$.

By $\mathcal E_{\rm comp}^{0,1}(M_a; X)$ we denote the class of $X$-valued $C^\infty$ (0,1)-forms on $E(S,\beta G)$ with compact supports in $M_a:=E(S,G_{in})$, i.e., $\omega\in \mathcal E_{\rm comp}^{0,1}(M_a; X)$ if there is a compact subset of $M_a$ such that in each local coordinate representation $\omega=f d\bar z$ support of $f$ belongs to it. By ${\rm supp}\,\omega$ we denote the minimal set satisfying this property. Let $\mathcal E_K^{0,1}(X)\subset \mathcal E_{\rm comp}^{0,1}(M_a; X)$ be the subspace of forms with supports in the compact set
$K\Subset M_a$.

Proofs of our main results are based on the following
\begin{Th}\label{dbar}
There exist a norm $\|\cdot\|_K$ on $\mathcal E_K^{0,1}(X)$ and a continuous linear operator $L_K: \bigl(\mathcal E_K^{0,1}(X),\|\cdot\|_K\bigr)\to \bigl(C(M(H^\infty);X),\sup_{M(H^\infty)}\|\cdot\|_X\bigr)$  such that
for each $\omega\in \mathcal E_{K}^{0,1}(X)$ 
\begin{itemize}
\item[(a)]
$L_K(\omega)|_{M_a}\in C^\infty(M_a;X)$\quad and\quad $\bar\partial(L_K(\omega)|_{M_a})=\omega$;
\item[(b)]
$L_K(\omega)|_{M(H^\infty)\setminus K}\in\mathcal O(M(H^\infty)\setminus K;X)$.
\end{itemize}
\end{Th}

\subsection{} 
In the proof of Theorem \ref{dbar} we use the following auxiliary results.

Let $B(z):=\prod_{j\ge 1}\frac{z-z_j}{1-\bar z_j z}$, $z\in\Di$, be an interpolating Blaschke product. According to \cite[Ch.~X, Lm.~1.4]{Ga} there exists $\varepsilon>0$ such that $B^{-1}(\Di_{\varepsilon}):=\sqcup_{j\ge 1}V_j$ and $B$ maps each $V_j$ biholomorphically onto $\Di_{\varepsilon}$.  By $b_j :\Di_\varepsilon\to V_j$ we denote the holomorphic map inverse to $B|_{V_j}$.
\begin{Proposition}\label{prop3.2}
There exists a positive $\delta\le\varepsilon$ and functions $f_j\in H^\infty(\Di\times\Di_\delta)$ such that
\begin{equation}\label{inter1}
f_j(b_j(w),w)=1,\quad f_j(b_k(w),w)=0,\quad k\ne j,
\end{equation}
\begin{equation}\label{inter2}
\sum_j |f_j(z,w)|\le 2M,\quad (z,w)\in \Di\times\Di_\delta,
\end{equation}
where 
\[
M:=\sup_{\|\{a_j\}\|_\infty\le 1}\inf\{\|f\|\, ;\, f\in H^\infty,\ f(z_j)=a_j,\ j=1, 2,\dots\}
\]
is the constant of interpolation for $\{z_j\}$.
\end{Proposition}
\begin{proof}
According to \cite[Ch.~VII, Th.~2.1]{Ga} there exist functions $g_j\in H^\infty$ such that
\[
g_j(z_j)=1,\quad g_j(z_k)=0,\quad k\ne j,\quad\text{and}\quad \sum_j |g_j(z)|\le M,\quad z\in \Di .
\]
Consider a bounded linear operator $L:\ell_\infty\to H^\infty$ of norm $\|L\|=M$ defined by the formula
\[
L(\{a_j\})(z):=\sum_j a_j g_j(z),\quad z\in\Di.
\]
Let $R(w)$ be the restriction operator to $\{b_j(w)\}$, $w\in\Di_\varepsilon$. Then
\[
(R(w)\circ L)(\{a_j\})(b_k(w)):=\sum_j a_j g_j(b_k(w)),\quad k\in\N.
\]
This and Cauchy estimates for derivatives of bounded holomorphic functions imply that $P(w):=R(w)\circ L:\ell_\infty\to\ell_\infty$, $w\in\Di_\varepsilon$, is a family of bounded operators  of norms $\le M$ holomorphically depending on $w$ and such that $P(0)={\rm id}$. The Cauchy estimates yield $\|\frac{dP}{dw}(w)\|\le\frac{M}{\varepsilon-|w|}$. In particular, for $|w|\le\delta:=\frac{\varepsilon}{3M}$ we have
\[
\bigl|\|P(w)\|-1\bigr|:=\bigl|\|P(w)\|-\|P(0)\|\bigr|\le |w|\frac{M}{\varepsilon-|w|}\le\frac 12.
\]
In particular, $P(w)$ is invertible and $\|P(w)^{-1}\|\le 2$. 

We set
\[
\hat L(w):=L\circ P(w)^{-1},\quad w\in\Di_\delta .
\]
Then $\hat L(w):\ell_\infty\to H^\infty$ is continuous, holomorphically depends on $w$ and $\|\hat L(w)\|\le 2M$. Moreover, $R(w)\circ\hat L(w)={\rm id}$.

We define
\begin{equation}\label{inter3}
f_j(\cdot,w):=\hat L(w)(\{\delta_{ij}\}),
\end{equation}
where $\delta_{ij}=1$ if $i=j$ and $0$ otherwise. Clearly $\{f_j\}$ satisfies the required properties.
\end{proof}

Let $\Gamma_B:=\{(z,w)\in\Di^2\, ;\, w=B(z)\}$ be the graph of $B$ in $\Di^2$.
For a complex Banach space $X$ and $\Di\times\Di_r\subset\Di^2$ consider the Banach space $C_{\rm comp}^{\rm h, u}(\Di\times\Di_r ;X)$ of bounded $X$-valued continuous functions on $\Di\times\Di_r$ with relatively compact images holomorphic with respect to the first coordinate and uniformly continuous with respect to the second one equipped with norm $\|f\|:=\sup_{(z,w)\in\Di\times\Di_r}\|f(z,w)\|_X$.
By $C_{\rm comp}^{\rm h, \infty}(\Di\times\Di_r ;X)\subset C_{\rm comp}^{\rm h, u}(\Di\times\Di_r ;X)$ we denote the subspace of $X$-valued functions having {\em bounded} derivatives of all orders with respect to the second coordinate. In addition, by $C_{\rm comp}(\Di_r\times\N; X)\supset
C_{\rm comp}^\infty(\Di_r\times\N ; X)$ we denote the space of $X$-valued continuous functions on $\Di_r\times\N$ with relatively compact images and its subspace of functions having bounded derivatives of all orders with respect to the coordinate in $\Di_r$. 
Let $R: f\mapsto f|_{\Gamma_B}$ be the restriction operator.
\begin{Proposition}\label{pr1}
For $\delta$ as in Proposition \ref{prop3.2} there exists a linear bounded operator $S: C_{\rm comp}(\Di_{\delta/2}\times\N ; X)\to C_{\rm comp}^{\rm h, u}(\Di\times\Di_{\delta/2} ;X)$ of norm $\le 2M$ which maps $C_{\rm comp}^\infty(\Di_{\delta/2}\times\N ; X)$ to $C_{\rm comp}^{\rm h, \infty}(\Di\times\Di_{\delta/2} ;X)$ and such that
\[
(R\circ S)(g)(b_j(w),w)=g(w,j)\quad\text{for all}\quad (w,j)\in\Di_{\delta/2}\times\N,\ g\in C_{\rm comp}(\Di_{\delta/2}\times\N ; X).
\]
\end{Proposition}
\begin{proof}
We define
\begin{equation}\label{soperator}
S(g)(z,w):=\sum_j f_{j}(z,w)g(w,j)
\end{equation}
with $f_j$ as in Proposition \ref{prop3.2}. Then the required result follows from properties \eqref{inter1}, \eqref{inter2} of that proposition.
\end{proof}

We retain notation of Propositions \ref{prop3.2} and  \ref{pr1}. 

By $E_\delta (X)$ we denote the space of $X$-valued $C^\infty$ $(0,1)$ forms $\omega=f d\bar z$ on $\Di$ with supports in $B^{-1}(\Di_{\delta/2})$ such that the function $\tilde f(w,j):=f(b_j(w))\frac{d \bar b_j(w)}{d\bar w}$, $(w,j)\in\Di_{\delta/2}\times\N$, belongs to $C_{\rm comp}^\infty(\Di_{\delta/2}\times\N ; X)$. 

\begin{Proposition}\label{pr2}
There exists a linear operator $G: E_\delta(X)\to C_{\rm comp}^\infty(\Di;X)$ such that for $\omega=f d\bar z\in E_\delta(X)$
\begin{itemize}
\item[(1)]
$$\bar\partial G(\omega)=\omega,$$ 
\item[(2)]
$$\sup_{z\in\Di}\|G(\omega)(z)\|_X\le 2M\cdot\sup_{(w,\,j)\in\Di_{\delta}\times\N}\|\tilde f(w,j)\|_X,$$
\item[(3)]
$$\{(G(\omega)\circ b_j)(w)\, ;\, (w,j)\in\Di_{\delta/2}\times\N\}\in C_{\rm comp}^\infty(\Di_{\delta/2}\times\N ; X);$$
\item[(4)]
$G(\omega)|_{\Di\setminus B^{-1}(\Di_{\delta/2})}$ is the uniform limit of a sequence of functions of the form $\sum_{i=0}^K h_iB^{-i}$ with $h_i\in H_{\rm comp}^\infty(\Di;X)$.
\end{itemize}
\end{Proposition}
\begin{proof}
We rewrite $\omega$ as a $(0,1)$ form on $\Di\times\N$ with values in $X$ replacing $z$ by $b_j(w)$ in each $V_j$. Then according to assumptions of the proposition we obtain the form $\tilde f d\bar w$ with $\tilde f\in C_{\rm comp}^\infty(\Di_{\delta/2}\times\N ; X)$. Consider the form $\tilde\omega:=S(\tilde f)d\bar w$, see \eqref{soperator}. Since ${\rm supp}\,\tilde\omega\subset\Di_{\delta/2}$, it can be regarded as a $C^\infty$ form on $\Di$ with values in $H^\infty(\Di; X)$. We define a linear operator $I :C_{\rm comp}^{\rm h, \infty}(\Di^2 ;X)\to C_{\rm comp}^{\rm h, \infty}(\Di^2 ;X)$ by the formula
\[
I(h)(z,w)=\frac{1}{2\pi i}\int\int_{\Di}\frac{h(z,\xi)}{\xi-w}d\xi\wedge d\bar\xi .
\]
(If we rewrite $I(h)(z,w)$ in polar coordinates as $\frac{1}{2\pi}\int\int_{w+\Di}h(z,re^{i\phi}+w)e^{-i\phi} dr\wedge d\phi$, then 
$I(h)$ has a relatively compact image because $h$ has it.)

Now, see, e.g., \cite[Ch.~VIII.1]{Ga},
\[
\frac{\partial I(h)}{\partial\bar w}=h\quad\text{and}\quad \sup_{(z,w)\in\Di^2}\|I(h)(z,w)\|_X\le \sup_{(z,w)\in\Di^2}\|h(z,w)\|_X.
\]
Finally, we set 
\[
G(\omega)(z):=I(S(\tilde f))(z,B(z)),\quad z\in\Di.
\]
Since $I(S(\tilde f))$ depends holomorphically on the first coordinate, 
$$
\begin{array}{c}
\displaystyle
\frac{\partial G(\omega)}{\partial\bar z}:=\frac{\partial I(S(\tilde f))(z,w)}{\partial\bar w}|_{w=B(z)}\cdot \frac{d\bar B(z)}{d\bar z}=S(\tilde f)(z, B(z))\frac{d\bar B(z)}{d\bar z}\\
\\
\displaystyle =
(f\circ b_j)(w)\frac{d\bar b_j(w)}{d\bar w}|_{w=B(z)}\cdot\frac{d\bar B(z)}{d\bar z}=f(z)\quad\text{on each}\quad V_j.
\end{array}
$$
(We have used that $b_j$ is the map inverse to $B$ on $V_j$.)

Next, properties (2) and (3) follow from the corresponding properties of the operator $I$. To prove (4) note that $I(S(\tilde f))(z,w)$, $z\in\Di$, $w\in \Co\setminus\Di_{\delta/2}$, can be regarded as a continuous up to the boundary bounded holomorphic function in $w$ with values in $H_{\rm comp}^\infty(\Di; X)$. Applying the Cauchy integral formula to this function (integrating over the boundary $\{w\in\Co\, ;\, |w|=\delta/2\}$) we approximate it uniformly on $\Co\setminus\Di_{\delta/2}$ by a sequence of functions of the form $\sum_{i=0}^K h_i(z)w^{-i}$, $(z,w)\in\Di\times \Co\setminus\Di_{\delta/2}$, $h_i\in H_{\rm comp}^\infty(\Di;X)$, $K\in\N$.
Replacing $w$ by $B(z)$ we get the required result.
\end{proof}

\subsection{}\begin{proof}[Proof of Theorem \ref{dbar}]
For $x\in K$ choose a coordinate chart $R_{U'(x),H'(x)}\Subset M_a$ containing it. Since $x$ is a closure of an interpolating sequence and the base of topology on $M_a$ is defined by means of interpolating Blaschke products (see subsection~2.1), there exist an interpolating
Blashke product $B_{s(x)}$ with zero set $s(x)$ whose closure contains $x$ and a number $\delta:=\delta(x)>0$ for which Proposition \ref{prop3.2} is valid for $B:=B_{s(x)}$ such that $\{y\in M_a\, ;\, |\hat B_{s(x)}(y)|<\delta(x)\}\subset R_{U'(x),H'(x)}$.

Further, choose a coordinate chart $R_{U(x),H(x)}\subset\{y\in M_a\, ;\, |\hat B_{s(x)}(y)|<\frac{\delta(x)}{2}\}$ containing $x$ such that
$U\Subset U'$. Let $\{R_{U(x_i),H(x_i)}\}_{1\le i\le k}$ be a finite subcover of the open cover  $\{R_{U(x),H(x)}\, ;\, x\in K\}$ of $K$.
By definition $\cup_{1\le i\le k}R_{U(x_i),H(x_i)}\subset E(S,\beta G)$ is an open neighbourhood of $K$. Consider an open cover of  $E(S,\beta G)$ consisting of sets $K^c:=E(S,\beta G)\setminus K$ and $R_{U(x_i),H(x_i)}$, $1\le i\le k$. Let $\varphi_0,\varphi_1,\dots,\varphi_k$ be a $C^\infty$ partition of unity subordinate to this cover (see Proposition \ref{partition}); here ${\rm supp}\,\varphi_i\Subset R_{U(x_i),H(x_i)}$, $1\le i\le k$, and ${\rm supp}\,\varphi_0\Subset K^c$. Now, since $K\cap ({\rm supp}\,\varphi_0)=\emptyset$, for any $\omega\in \mathcal E_K^{0,1}(X)$ we have
\[
\omega=\sum_{j=0}^k\varphi_j\omega=\sum_{j=1}^k\varphi_j\omega .
\]
Note that ${\rm supp}\,(\varphi_i\omega)\subset K_i:=K\cap \bar R_{U(x_i),H(x_i)}$, $1\le i\le k$. Clearly, it suffices to prove the theorem for spaces $\mathcal E_{K_i}^{0,1}(X)$ and then define for $\omega\in \mathcal E_K^{0,1}(X)$,
\begin{equation}\label{reduct}
L_K(\omega):=\sum_{i=1}^k L_{K_i}(\varphi_i\omega),\quad\text{and}\quad \|\omega\|_K:=\sum_{i=1}^k\|\varphi_i\omega\|_{K_i}. 
\end{equation}

Therefore from now on we will assume that 
$$
\begin{array}{l}
\displaystyle
{\rm supp}\,\omega\subset K\subset R_{U,H}\subset\left\{y\in M_a\, ;\, |\hat B(y)|<\frac{\delta}{2}\right\}
\\
\displaystyle \subset\{y\in M_a\, ;\, |\hat B(y)|<\delta\}\subset R_{U',H'}, \quad U\Subset U',
\end{array}
$$
where $B$ is an interpolating Blashke product with $\delta$ as in Proposition \ref{prop3.2}. 

Without loss of generality we will identify $U'$ with $\Di$ and $U$ with $\Di_t$ for some $t<1$. As before for some $(\delta\le)\, \varepsilon<1$ by $b_j:\Di_\varepsilon\to\Di$ we denote the map inverse to $B$ on $V_j$, where $B^{-1}(\Di_\varepsilon)=\sqcup_j V_j$.

Consider the restriction $\omega|_{\Di}=f(s,g)d\bar s$, where ${\rm supp}\, f\subset R_{U,H}\cap\Di$, $s$ is a holomorphic coordinate in $U'\subset S$ and $g\in G$. Since $\omega$ is of class $C^\infty$ and has compact support, $f\in C_{\rm comp}^\infty(U'\times G; X)$. 
Substituting $s=r(z),\, z\in\Di$ ($\, r:=\tilde r|_{\Di}$), we have
$\omega|_\Di=f(r(z),g)\frac{d\bar r}{d\bar z}d\bar z=:F(z)d\bar z$. Here $F$ is an $X$-valued $C^\infty$ function on $\Di$ with support in $r^{-1}(U)\cap \{z\in\Di\, ;\, |B(z)|<\frac{\delta}{2}\}$. Also,
\begin{equation}\label{chain}
(F\circ b_j)(w)\frac{d\bar b_j(w)}{d\bar w}=f((r\circ b_j)(w),g_j)\cdot\frac{d(\overline{r\circ b_j})(w)}{d\bar w},\quad w\in\Di_\delta ,
\end{equation}
where $g_j\in G$ is uniquely defined by the condition $b_j(w)\in V_j$.

According to our assumption, each element of the family of holomorphic functions $\{r\circ b_j\}_{j\in\N}$ maps $\Di_\delta$ into $\Di\, (:=U')$. In particular, any order derivatives of elements of this family are uniformly bounded on $\Di_{\delta/2}$. Since $f\in C_{\rm comp}^\infty(U'\times G;X)$, this and \eqref{chain} imply that the function $\widetilde F(w,j):=(F\circ b_j)(w)\frac{d\bar b_j(w)}{d\bar w}$, $(w,j)\in\Di_{\delta/2}\times\N$, belongs to $C_{\rm comp}^\infty(\Di_{\delta/2}\times\N\, ;\, X)$. This means that $\omega|_{\Di}\in E_\delta (X)$ and we can apply Proposition \ref{pr2}. 
Let $h:=G(\omega)\in C_{\rm comp}^\infty(\Di ;X)$ be the function satisfying conditions (1)--(4) of this proposition. Show that
it can be continuously extended to $V:=\{y\in M_a\, ;\, |\hat B(y)|<\frac{2\delta}{3}\}$ so that the extension is of class $C^\infty$.

In fact, $\omega=f(s,x)d\bar s$ on $\tilde r^{-1}(U')\, (\cong U'\times \beta G\ni (s,x))$ with ${\rm supp}\, f\subset R_{U,H}$.
Since $f\in C^\infty(\tilde r^{-1}(U');X)$, we can solve equation $\bar\partial H=\omega$ on $\tilde r^{-1}(U')$ by the formula:
\[
H(s,x)=\frac{1}{2\pi i}\int\int_{U'}\frac{f(\xi,x)}{\xi-s}d\xi\wedge d\bar\xi .
\]
In particular, $H\in C_{{\rm comp}}^\infty(\tilde r^{-1}(U'); X)$. Therefore,
$c:=h|_{r^{-1}(U')}-H|_{r^{-1}(U')}$ is an $X$-valued bounded holomorphic function on $r^{-1}(U')\cong U'\times G$ with relatively compact support (see Proposition \ref{pr2}\ (1)).
As it follows from \cite[Th.~2.1]{Br1} function $c$ admits a continuous extension $\hat c:\tilde r^{-1}(U')\to X^{**}$ (where $X^{**}$ is considered in the weak$*$ topology) such that $\varphi(\hat c(\cdot, x))$ is holomorphic for all $\varphi\in X^*$ and $x\in\beta G$. But ${\rm Im}(\hat c)\subset \overline{{\rm Im}(c)}\Subset X$ (here \, $\bar{\, }$\, stands for the weak$*$ closure in $X^{**}$). Thus, since the weak$*$ topology is equivalent to the norm-topology on each compact subset of $X$, $\hat c\in C_{{\rm comp}}(\tilde r^{-1}(U');X)$ and each $\hat c(\cdot, x)$, $x\in\beta G$, is an $X$-valued holomorphic function on $U'$ . Also, the family 
$\{\hat c(\cdot, x)\, ;\ x\in\beta G\}$ is uniformly bounded on $U'$. 
Applying the Cauchy estimates for derivatives of bounded holomorphic functions to elements of this family we obtain that $\hat c|_{\tilde r^{-1}(U'')}\in C_{\rm comp}^\infty(\tilde r^{-1}(U'');X)$ for any open $U''\Subset U'$. Finally, since $V\Subset R_{U',H'}$, there exists $U''\Subset U'$ such that $V\subset \tilde r^{-1}(U'')$. This shows that
$h$ is extended to $\Di\cup V$ as a $C^\infty$ function with relatively compact image.

Further, closure of the set $W:=\{z\in\Di\, ;\, |B(z)|> \frac{\delta}{2}\}$ contains $M(H^\infty)\setminus\{\Di\cup V\}$. According to Proposition \ref{pr2}~(4), $h|_{W}$ is the uniform limit of a sequence of functions of the form $\{\sum_{i=0}^K h_i B^{-i}\, ;\, h_i\in H_{\rm comp}^\infty(\Di; X)\}_{K\in\N}$. Clearly, each $h_i$ is extended to $M(H^\infty)$ as an $X$-valued holomorphic function (cf. similar arguments for $c$). Thus, $h|_{W}$ is extended to  $M(H^\infty)\setminus\{\Di\cup V\}$ as a continuous function holomorphic in interior points of this set  (in particular, this extension is of class $C^\infty$ there).

We conclude that $h$ is continuously extended to $M(H^\infty)$ and the extension $\hat h$ is of class $C^\infty$ on $M_a$. Since $\bar\partial h=\omega|_{\Di}$, we obtain by continuity of derivatives of $\hat h$ that $\bar\partial\hat h=\omega$ on $M_a$. Clearly, $\hat h$ is holomorphic outside ${\rm supp}\,\omega$. Finally, we set $L_{K}(\omega):=\hat h$ and $\|\omega\|_{K}:=\sup_{(w,j)\in \Di_\delta\times\N}\|\widetilde F(w,j)\|_X$. Then the required properties of $L_K$ follow from Proposition \ref{pr2}.

The proof of the theorem is complete. 
\end{proof}

\sect{Proofs of Proposition \ref{sheaf} and Theorems \ref{main}, \ref{cousin}}
\subsection{Proof of Proposition \ref{sheaf}}
First, we prove the result for $B=\Co$. 

Let us consider an open cover $(U_i)_{i\in I}$ of open $U\subset M(H^\infty)$ such that $U_i\Subset U$, $i\in I$, and its refinement
$(V_j)_{j\in J}$ such that $V_j\subset\bar V_j\subset U_{\tau(j)}$, where $\tau: J\to I$ is the refinement map. Let $f\in \mathcal O(U\cap\Di)$. Then according to \cite[Th.~3.2]{S1} applied to sets $V_j, U_{\tau(j)}$ and the function $f|_{U_{\tau(j)\cap\Di}}\in H^\infty(U_{\tau(j)}\cap\Di)$, there exists a family of functions $\tilde f_j\in C(\bar V_j)$ such that $\tilde f_j(z)=f(z)$ on $V_j\cap\Di$, $j\in J$. If now $V_j\cap V_k\ne\emptyset$, then
$\tilde f_j-\tilde f_k=0$ on the closure $\overline{V_j\cap V_k\cap\Di}\subset M(H^\infty)$ of $V_j\cap V_k\cap\Di$. But this compact set contains $V_j\cap V_k$. (For otherwise, there exists an open subset $W$ of $V_j\cap V_k$ such that $W\cap \overline{V_j\cap V_k\cap\Di}=\emptyset$. But $\Di$ is dense in $M(H^\infty)$ and so $W\cap\Di\ne\emptyset$. On the other hand, $W\cap\Di\subset V_j\cap V_k\cap\Di=\emptyset$, a contradiction.) Thus $\tilde f_j=\tilde f_k$ on $V_j\cap V_k$. This implies that there exists a function
$\tilde f\in\mathcal O(U)$ such that $\tilde f|_{U\cap\Di}=f$ defined by $\tilde f:=\tilde f_j$ on $V_j$, $j\in J$. Clearly, such $\tilde f$ is unique.

Let us consider the general case. Assume that $f\in\mathcal O(U\cap \Di;B)$ is such that $f|_{V\cap\Di}\in H_{\rm comp}^\infty(V\cap\Di; B)$ for every open $V\Subset U$. Applying the scalar case of the proposition to the family of functions $\varphi\circ f$ with $\varphi\in B^*$ we obtain that $f$ has an extension $\tilde f\in \mathcal O(U; B^{**})$. However, for each open $V\Subset U$, the image of $\tilde f|_{V}$ belongs to the weak$^*$ closure in $B^{**}$ of the compact subset $\overline{f(V\cap\Di)}$ of $B$. (We have used here that $V\subset\overline{V\cap\Di}$.) The former set being compact is weak$^*$ closed. Thus $\tilde f|_V\in H_{\rm comp}^\infty(V;B)$ for all such $V$. This implies that $\tilde f\in \mathcal O(U;B)$ as required.\qquad $\Box$

\subsection{} The following auxiliary result will be used in the proof of Theorem \ref{main}.
\begin{Lm}\label{le4.2}
For an open cover of $M(H^\infty)$ there exist finite open refinements $(W_j)_{1\le j\le k}$  and $(V_j)_{1\le j\le k}$  
such that
\begin{enumerate}
\item
$\bar V_j\Subset W_j$ for all $1\le j\le k$; 
\item $V_{1},\dots, V_{r}$ cover $M_s$ and $\bar W_{m}\cap\bar W_{n}=\emptyset$ for all $1\le m\ne n\le r$; 
\item
$W_{r+1},\dots, W_{k}$ are relatively compact subsets of $M_a$. 
\end{enumerate}
\end{Lm}
\begin{proof}
Since $M_s=0$ is compact and totally disconnected, for each cover $(U_i)$ of $M_s$ by open subsets of $M(H^\infty)$ there exists a finite open refinement $(U_j')$ such that each $U_j'$ is relatively compact in all $U_i$ containing it and $\bar U_m'\cap \bar U_n'=\emptyset$ for all $m\ne n$. 
Thus for an open cover $\mathcal V$ of $M(H^\infty)$ there exists a finite refinement $(V_j)_{1\le j\le k}$ such that $V_{1},\dots, V_{r}$ cover $M_s$ and $\bar V_{m}\cap\bar V_{n}=\emptyset$, $1\le m\ne n\le r$, and $V_{r+1},\dots, V_{k}$ are relatively compact coordinate charts in $M_a$. This implies that there exist open $W_1,\dots, W_k$ such that $V_j\Subset W_j$ for all $j$, $\bar W_m\cap\bar W_n$ for $1\le m\ne n\le r$,  $W_j\Subset M_a$ for $r+1\le j\le n$, and $W_1,\dots, W_k$ is a refinement of $\mathcal V$ as well.
\end{proof}
\noindent {\bf Partition of unity.}
Recall that $q:E(S,\beta G)\to M(H^\infty)$ is the quotient map which is identity on $E(S, G_{in})=M_a$, see subsection~2.1.  In notation of Lemma \ref{le4.2} we set $V_j':=q^{-1}(V_j)$, $W_j':=q^{-1}(W_j)$. By $\psi_j$ we denote $C^\infty$ functions on $E(M,\beta G)$ such that $\psi_j|_{\bar V_j'}>0$ and ${\rm supp}\,\psi_j\subset W_j'$ (see Lemma \ref{le4.1}). Then $\sum_{j=1}^k\psi_j(x)>0$ for each $x\in E(S,\beta G)$ (because $(V_j')$ is a cover of $E(S,\beta G)$). We define 
\[
\varphi_j:=\frac{\psi_j}{\sum_{j=1}^k\psi_j}.
\]
Then $\varphi_j$ is a nonnegative $C^\infty$ function on $E(S,\beta G)$, ${\rm supp}\, \varphi_j\subset W_j'$, $\varphi_j=1$ on $\widetilde V_j:=V_j'\setminus\bigl(\cup_{i\ne j} \bar W_i'\bigr)$ and $\sum_{j=1}^k\varphi_j=1$. Observe that $(\widetilde V_j)_{j=1}^r$ is an open cover of $q^{-1}(M_s)$ by pairwise disjoint open sets.
\subsection{Proof of Theorem \ref{main}}
We retain notation of the previous subsection.

\begin{proof}[Proof of the theorem for $k=1$] 

Let $c:=\{c_{ij}\}\in Z^1(\mathcal W ; \mathcal O_{M(H^\infty)}^B)$ be a cocycle defined on an open cover $\mathcal W$ of $M(H^\infty)$. Passing to a refinement of $\mathcal W$ we may assume that $\mathcal W=(W_j)_{1\le j\le k}$ is as in Lemma \ref{le4.2}. 
According to our construction if $W_i\cap W_j\ne\emptyset$ for $i\ne j$, then this set belongs to $M_a$. Using the above constructed $C^\infty$ partition of unity $\{\varphi_j\}$ subordinate to the cover $(W_j')$  we resolve $c$ by the formulas
\[
h_i:=\sum_k \varphi_k c_{ik}\quad \text{on}\quad W_i.
\]
Here the sum is taken over all $k$ for which $W_i\cap W_k\ne\emptyset$. 

Indeed, since  
${\rm supp}\, \varphi_k\subset W_k'$ and $q: E(S, G_{in})\to M_a$ is the identity map, each function $\varphi_k c_{ik}$ can be thought of as a continuous $B$-valued function on $W_i$ with support in $W_i\setminus M_s$ and of class $C^\infty$ on this set. Thus every $h_i$ also satisfies these properties and $h_i-h_j=c_{ij}$ on $W_i\cap W_j\ne\emptyset$. In particular, we can define a $B$-valued $C^\infty$ $(0,1)$ form $\omega$ on $M_a$ by the formulas
\[
\omega=\bar\partial h_i\quad\text{on}\quad W_i.
\] 
Since $\bigl({\rm supp}\, h_i\bigr)\cap M_s=\emptyset$ for all $i$, the form $\omega$ has compact support in $M_a$. Hence, it can be regarded as a form on $E(S,\beta G)$ from the class $\mathcal E_{\rm comp}^{0,1}(M_a;B)$. Now,
according to Theorem \ref{dbar}, there exists a function $h$ on $M(H^\infty)$ such that $\bar\partial h=\omega$ on $M_a$ and $h$ is holomorphic outside ${\rm supp}\,\omega$. We define
\[
c_i:=h_i-h\quad\text{on}\quad W_i.
\]
Then $c_i\in\mathcal O(W_i; B)$ and $c_i-c_j=c_{ij}$ on $W_i\cap W_j$. This shows that cocycle $c$ determines a zero element of  $H^1(M(H^\infty); \mathcal O_{M(H^\infty)}^B)$.

\medskip

\noindent{\em Proof of the theorem for $k=2$.} We will use the following result.
\begin{Lm}\label{rest}
Let $U\Subset V\subset M_a$ be open and $B$ be a complex Banach space. Assume that $\omega$ is a $C^\infty$ $B$-valued $(0,1)$-form on $V$. Then there exists a continuous $B$-valued function $g$ on $M(H^\infty)$ of class $C^\infty$ on $V$ such that $\bar\partial h=\omega$ on $U$. 
\end{Lm}
\begin{proof}
Let $\rho$ be a nonnegative $C^\infty$ function on $E(S,\beta G)$ equals $1$ in an open neighbourhood of $\bar U$ with ${\rm supp}\,\rho\subset V$,
see Corollary \ref{cutoff1}. Then according to Theorem \ref{dbar} the function $g:=L_{{\rm supp}\,\rho}(\rho\cdot\omega)$ satisfies the required properties.
\end{proof}

Let $c:=\{c_{ijk}\}\in Z^2(\mathcal W ; \mathcal O_{M(H^\infty)}^B)$ be a cocycle defined on an open cover $\mathcal W$ of $M(H^\infty)$. Passing to a refinement of $\mathcal W$ we may assume that $\mathcal W=(W_j)_{1\le j\le k}$ is as in Lemma \ref{le4.2}; 
here $W_i\cap W_j\ne\emptyset$ for $i\ne j$ implies that this set belongs to $M_a$. Using the above constructed $C^\infty$ partition of unity $\{\varphi_j\}$ subordinate to the cover $(W_j')$  we resolve $c$ by the formulas
\[
h_{ij}:=\sum_k \varphi_k c_{ijk}\quad \text{on}\quad W_i\cap W_j\ne\emptyset.
\]
Here the sum is taken over all $k$ for which $W_i\cap W_j\cap W_k\ne\emptyset$. 

Now the family $\{\bar\partial h_{ij}\}$ of $B$-valued $(0,1)$-forms is a $1$-cocycle on ${\mathcal W}$. Resolving it we obtain the family of $B$-valued $(0,1)$-forms
\[
\omega_i:=\sum_k \varphi_k \bar\partial h_{ik}\quad \text{on}\quad W_i\quad\text{with}\quad {\rm supp}\,\omega_i\Subset \bar W_i\cap M_a.
\]
Let $\widetilde{\mathcal W}=(\widetilde W_j)$ be a finite refinement of $\mathcal W$ such that $\widetilde W_j\Subset W_{\tau(j)}$; here $\tau$ is the refinement map defined on the set of indices of $\widetilde{\mathcal W}$. According to Lemma \ref{rest} there exists a $B$-valued continuous function $g_j$ on $\widetilde W_j$ such that
$\bar\partial g_j=\omega_{\tau(j)}$ on $\cup_{i\ne j} \widetilde W_j\cap \widetilde W_i$. By definition,
\[
\bar\partial(g_i-g_j)=\bar\partial h_{\tau(i)\tau(j)}|_{\widetilde W_i\cap\widetilde W _j}\quad\text{on}\quad \widetilde W_i\cap\widetilde W_j\ne\emptyset.
\]
In particular, $c_{ij}:=h_{\tau(i)\tau(j)}|_{\widetilde W_i\cap\widetilde W _j}-(g_i-g_j)\in \mathcal O(\widetilde W_i\cap\widetilde W _j;B)$
and
\[
c_{ij}-c_{jk}+c_{ki}=c_{\tau(i)\tau(j)\tau(k)}\quad\text{on}\quad \widetilde W_i\cap \widetilde W_j\cap \widetilde W_k.
\]

This implies that $\{c_{ijk}\}$ represents zero element of $H^2(M(H^\infty) ; \mathcal O_{M(H^\infty)}^B)$.

\medskip

\noindent{\em Proof of theorem for $k\ge 3$.} According to \cite{S1} ${\rm dim}\, M(H^\infty)=2$. Therefore $H^k(M(H^\infty);\mathcal J)=0$ for any sheaf ${\mathcal J}$ of abelian groups on $M(H^\infty)$ and all $k\ge 3$.
\end{proof}

\subsection{Proof of Theorem \ref{cousin}}
\begin{proof}
Let $\mathcal O_{M(H^\infty)}^*, C_\Z$ be sheaves of germs of nowhere vanishing holomorphic functions and of integer-valued continuous functions on $M(H^\infty)$. Then we have the short exact sequence of sheaves
\[
0\rightarrow C_\Z\rightarrow \mathcal O_{M(H^\infty)}\stackrel{2\pi i\cdot\exp}{\longrightarrow}\mathcal O_{M(H^\infty)}^*\to 1.
\] 
The corresponding long cohomology sequence has the form
\[
\dots\rightarrow H^1(M(H^\infty);\mathcal O_{M(H^\infty)})\rightarrow H^{1}(M(H^\infty);\mathcal O_{M(H^\infty)}^*)\to H^2(M(H^\infty);C_\Z)\rightarrow\dots\ .
\]
From \cite[Cor.~3.9]{S1} one obtains $H^2(M(H^\infty);C_\Z)=0$. Together with Theorem \ref{main} this implies triviality of $H^{1}(M(H^\infty);\mathcal O_{M(H^\infty)}^*)$. Now, any divisor $D=\{(U_i, h_i)\}_{i\in I}$ on $M(H^\infty)$ determines a $1$-cocycle $\{\frac{h_i}{h_j}\}\in Z^1((U_i); \mathcal O_{M(H^\infty)}^*)$. Since the corresponding cohomology class is trivial, there exists a refinement $(V_j)$ of $(U_i)$ and holomorphic functions $c_j\in \mathcal O^*(V_j)$ such that $c_i^{-1}c_j=(h_{\tau(i)} h_{\tau(j)}^{-1})|_{V_i\cap V_j}$, where $\tau$ is a map from the set of indices of $(V_j)$ into the set of indices of $(U_i)$ such that $V_j\subset U_{\tau(j)}$. Therefore the formulas
\[
h|_D:=c_i\cdot (h_{\tau(i)}|_{V_i})\quad\text{on}\quad V_i
\]
determine a meromorphic function on $M(H^\infty)$ such that $(h|_D)|_{U_i}\cdot h_i^{-1}\in \mathcal O^*(U_i)$ for all $i$. 
\end{proof}

\sect{Proof of Theorem \ref{runge}}
We start with some auxiliary results. 
\begin{Lm}\label{polyhedr}
Let $N\subset M(H^\infty)$ be a neigbourhood of a holomorphically compact set $K$. Then there exists a polyhedron containing in $N$ whose interior
contains $K$. Thus it suffices to prove the theorem for $K$ a polyhedron.
\end{Lm}
\begin{proof}
For each $x\in K^c:=M(H^\infty)\setminus K$ by $f_x\in H^\infty$ we denote a function such that $\max_K |f_x|<1<|f_x(x)|$. Let ${\mathcal F}:=\{\Pi\bigl((f_{x_{\ell_1}},\dots, f_{x_{\ell_\alpha}})\bigr)\, ;\, x_{\ell_j}\in K^c, 1\le j\le\alpha\}_{\alpha\in\Lambda}$ be the family of all polyhedra formed by functions in $\{f_x\}_{x\in K^c}$.
If any polyhedron from ${\mathcal F}$ has a nonempty intersection with $N^c$, then intersection of all polyhedra of this family has a nonempty intersection with $N^c$ as well (because each polyhedron and $N^c$ are compact). On the other hand, by the choice of $f_x$ this intersection coincides with $K$, a contradiction. Thus, there exists a polyhedron $\Pi\bigl((f_{y_{1}},\dots, f_{y_{k}})\bigr)\subset N$. By the definition of functions $f_{y_j}$ the set $K$ belongs to the interior of this polyhedron.
\end{proof}

From now on we will assume that $K:=\Pi\bigl((f_{y_{1}},\dots, f_{y_{k}})\bigr)\subset N$.

\begin{Lm}\label{twosets}
There are open subsets $U\Subset V\subset N$ of $M(H^\infty)$ such that $K\subset U$ and $U\cap M_s=\bar V\cap M_s$.
\end{Lm}
\begin{proof}
Without loss of generality we may assume that $\emptyset\ne K\subset N\ne M(H^\infty)$ (for otherwise the statement of the lemma holds either with $U=V:=\emptyset$ or with $U=V:=M(H^\infty)$). 
Next, since $M_s$ is totally disconnected, the family of all clopen subsets of $M_s$ forms a base of topology of $M_s$. Hence, compactness of $K\cap M_s$ implies that there exists a clopen set $O\Subset N\cap M_s$ containing $K\cap M_s$. For each $x\in O$ consider open subsets $U_x\Subset V_x\subset N$ of $M(H^\infty)$ such that $\bar V_x\cap M_s\subset O$. Since $O$ is compact, there is a finite subcover $\{U_{x_j}\}_{1\le j\le\ell}$ of $\{U_x\}_{x\in O}$ covering $O$.  We set $U_*:=\cup_{1\le j\le\ell} U_{x_j}$ and $V_*:=\cup_{1\le j\le\ell} V_{x_j}$. Clearly, $U_*\cap M_s=\bar V_*\cap M_s=O$ and $U_*\Subset V_*\subset N$. 

Further, consider $K':=K\setminus U_*$. By definition, $K'$ is compact and $K'\cap M_s=\emptyset$ (indeed, $K'\cap M_s\subset \bigl(K\cap M_s\bigr)\setminus U_*\subset O\setminus U_*\subset U_*\setminus U_*=\emptyset$). Thus as before there are open covers $(U_j')_{1\le j\le k}$ and $(V_j')_{1\le j\le k}$ of $K'$ such that $U_j'\Subset V_j'\subset N$ and $\bar V_j'\cap M_s=\emptyset$ for all $j$. We define $U:=U_*\cup (\cup_{1\le j\le k} U_j')$ and $V:=V_*\cup (\cup_{1\le j\le k} V_j')$. Clearly these sets satisfy the required properties.
\end{proof}
For subsets $U$ and $V$ of Lemma \ref{twosets} with $\emptyset\ne K\subset N\ne M(H^\infty)$ we set
$U_q:=q^{-1}(U)$ and $V_q:=q^{-1}(V)$. These are proper nonempty open subsets of $E(S,\beta G)$ such that $U_q\Subset V_q$. 
\begin{Lm}\label{cutoff}
There exists a real $C^\infty$ function $\varphi$ on $E(S,\beta G)$ equals $1$ in an open neighbourhood of $\bar U_q$ with ${\rm supp}\,\varphi\subset V_q$. Moreover, ${\rm supp}\,\bar\partial\varphi$ is a nonempty compact subset of $M_a$.
\end{Lm}
\begin{proof}
Existence of such a function $\varphi$ follows from Corollary \ref{cutoff1}. Let us check the second statement for $\varphi$.

Note that $A:=\{x\in E(S,\beta G)\,;\, \varphi(x)\notin\{0,1\}\}\Subset V_q\setminus \bar U_q$ and is nonempty. (Indeed, $E(S,\beta G)$ is connected, and so $\varphi(E(S,\beta G))=[0,1]$.)
Now, from the fact that $q^{-1}(U\cap M_s)=q^{-1}(\bar V\cap M_s)$ we get 
\[
(V_q\setminus\bar U_q)\cap q^{-1}(M_s)\subset \bigl(V_q\cap q^{-1}(M_s)\bigr)\setminus\bigl(U_q\cap q^{-1}(M_s)\bigr)=q^{-1}(V\cap M_s)\setminus q^{-1}(U\cap M_s)=\emptyset.
\]
Therefore, $A\Subset E(S,G_{in})=M_a$. This implies that ${\rm supp}\,\bar\partial\varphi\subset A\Subset M_a$. Assuming that
${\rm supp}\,\bar\partial\varphi=\emptyset$ we obtain that $\bar\partial\varphi|_{\Di}=0$. That is, $\varphi|_{\Di}$ is a nonnegative $C^\infty$ holomorphic function on $\Di$. Hence, it is a constant, a contradiction.
\end{proof}
\begin{proof}[Proof of Theorem \ref{runge}]
We retain notation of Lemmas \ref{polyhedr}--\ref{cutoff}. 

According to Lemma \ref{cutoff} the function $\max_{1\le j\le k}|f_{y_j}|$ is greater than $1$ on ${\rm supp}\,\bar\partial\varphi$. Let $W\Subset M_a$ be an open neighbourhood of ${\rm supp}\,\bar\partial\varphi$. Then
${\rm supp}\,\bar\partial\varphi$ is covered by open sets $U_j:=\{z\in M(H^\infty)\,;\, |f_{y_j}(z)|>r\}\cap W$, $1\le j\le k$, for some $r>1$.
Using Proposition \ref{partition} we find real nonnegative $C^\infty$ functions $\varphi_j$ on $E(S,\beta G)$ such that ${\rm supp}\,\varphi_j\subset U_j$, $1\le j\le k$, and $\sum_{1\le j\le k}\varphi_j=1$ in an open neighbourhood of ${\rm supp}\,\bar\partial\varphi$.

Now, suppose that $g\in \mathcal O(N;B)$. Without loss of generality we may assume that $\emptyset\ne K\subset W\ne M(H^\infty)$.
For pullback $q^*g\in\mathcal O(q^{-1}(N);B)$ of $g$ consider $B$-valued $(0,1)$ differential forms 
$\omega_j:=(q^*g)\cdot\varphi_j\cdot \bar\partial\varphi$ on $E(S,\beta G)$; here ${\rm supp}\,\omega_j\subset\bigl({\rm supp} \,\bar\partial\varphi\bigr)\cap U_j\Subset M_a$.   
For each $n\in\N$, $1\le j\le k$, we set
\begin{equation}\label{omn}
\omega_{jn}:=\frac{\omega_j}{(q^*f_{y_j})^n}.
\end{equation}
Then $\omega_{jn}$ is a $B$-valued $C^\infty$ $(0,1)$-form on $E(S,\beta G)$ and ${\rm supp}\,\omega_{jn}={\rm supp}\,\omega_j=:S_j$. 
Since  $|q^*f_{y_j}|>r>1$ on $S_j$, \eqref{omn} implies that $\lim_{n\to\infty}\|\omega_{jn}\|_{S_j}=0$, see the proof of Theorem \ref{dbar}. Applying this theorem we find functions $h_{jn}:=L_{S_j}(\omega_{jn})$ on $M(H^\infty)$ of class $C^\infty$ on $M_a$ such that
$\bar\partial h_{jn}=\omega_{jn}$ there. From here and \eqref{omn} we obtain
\[
\bar\partial\left(\sum_{j=1}^k f_{y_j}^n\cdot h_{jn}\right)=\sum_{j=1}^k \omega_j= g\bar\partial\varphi=\bar\partial(g\varphi)\quad\text{on}\quad M_a.
\]
Hence,
\[
\tilde g_n:=(q^*g)\cdot\varphi-\sum_{j=1}^k (q^*f_{y_j})^n\cdot q^*h_{jn}
\]
is a $B$-valued continuous function on $E(S,\beta G)$ holomorphic on $M_a$. This means that $\tilde g_n\in\mathcal O(M(H^\infty);B)$. Therefore there is $g_n\in H^\infty\, (:=\mathcal O(M(H^\infty);B))$ such that $q^*g_n=\tilde g_n$.
Since the first term in the definition of $\tilde g_n$ coincides with $q^*g$ on $q^{-1}(K)$, $\max_{1\le j\le k}|f_{y_j}^n|\le 1$ on $K$ and $\lim_{n\to\infty}\sup_{x\in M(H^\infty)}\{\max_{1\le j\le k}\|h_{jn}(x)\|_{B}\}\le C\cdot\lim_{n\to\infty}\max_{1\le j\le k}\|\omega_{jn}\|_{S_j}=0$, the sequence of functions $\{g_n\}$ converges uniformly on $K$ to $g$, as required.
\end{proof}

\sect{Proofs of Theorems \ref{th2}, \ref{th3}, \ref{square} and Corollary \ref{cor1}}
\begin{proof}[Proof of Theorem \ref{th2}]
If $f$ belongs to the ideal $I\subset \mathcal O(M(H^\infty);A)$ generated by $f_1,\dots, f_m$, then as the open cover of the theorem we can take $M(H^\infty)$. So in this direction the result is trivial. Let us prove the converse statement.

Let $\mathcal U=(U_j)_{1\le j\le\ell}$ be a finite open cover of $M(H^\infty)$ satisfying assumptions of the theorem. Passing to a refinement of $\mathcal U$ we may replace it by a cover $(W_j)_{1\le j\le k}$ satisfying conditions of Lemma \ref{le4.2}. Then by $\{\varphi_j\}$ we denote a $C^\infty$ partition of unity subordinate to the cover $(W_j')_{1\le j\le k}$ of $E(S,\beta G)$; here $W_j':=q^{-1}(W_j)$. By the assumption of the theorem there exists a family of functions $g_{ij}\in H_{\rm comp}^\infty (W_j; A)$, $1\le i\le m$, $1\le j\le k$, such that
\begin{equation}\label{e5.5}
f|_{W_j}=\sum_{i=1}^m g_{ij}f_i\quad\text{on each}\quad W_j .
\end{equation}
We set 
\[
c_{i,rs}:=g_{ir}-g_{is}\quad\text{on}\quad W_r\cap W_s\ne\emptyset ,
\]
and then
\[
h_{ir}:=\sum_{s} \varphi_s c_{i,rs},
\]
where the sum is taken over all $s$ for which $W_s\cap W_r\ne\emptyset$.
Since in this case $W_s\cap W_r\subset M_a$, functions $h_{ir}\in C_{\rm comp}^\infty(W_r;A)$ and ${\rm supp}\, h_{ir}\Subset \bar W_r\cap M_a$, see the proof of Theorem \ref{dbar} for similar arguments. Moreover,
\[
h_{ir}-h_{is}=c_{i,rs}\quad\text{on}\quad W_r\cap W_s\ne\emptyset.
\]
Further, define
\[
h_i:=g_{ir}-h_{ir}\quad\text{on}\quad W_r .
\]
Clearly, each $h_i$ is a continuous $A$-valued function on $M(H^\infty)$, holomorphic in an open neighbourhood of $M_s$ and of class $C^\infty$ on $M_a$. In particular, $\bar\partial h_i$ can be regarded as $C^\infty$ $A$-valued $(0,1)$ forms on $E(S,\beta G)$ with (compact) supports in $M_a$. 

Also, according to \eqref{e5.5} we have
\[
f=\sum_{i=1}^m h_i f_i .
\]

Now we apply arguments similar to those in \cite[Ch.~VIII, Th.~2.1]{Ga}. Specifically, if we set
\[
g_i:=h_{i}+\sum_{s=1}^m a_{is}f_s,
\]
where
\begin{equation}\label{eq5.6'}
a_{is}=b_{is}-b_{si}\quad\text{and}\quad \bar\partial b_{is}=h_i\bar\partial h_s ,
\end{equation}
then 
\[
\sum_{i=1}^m g_if_i=f\quad\text{and}
\]
\[
\bar\partial g_i=\bar\partial h_i+\sum_{s=1}^m f_s\cdot (h_i\bar\partial h_s- h_s\bar\partial h_i)=\bar\partial h_i+h_i\bar\partial\left(\sum_{s=1}^m f_s h_s\right)-\bar\partial h_i\sum_{s=1}^m f_s h_s=0.
\]
Hence, $g_i\in \mathcal O(M(H^\infty);A)$ and so $f$ belongs to the ideal $I\subset \mathcal O(M(H^\infty);A)$ generated by $f_1,\dots, f_m$.

To complete the proof it remains to solve equations $\bar\partial b_{is}=h_i\bar\partial h_s$ on $M(H^\infty)$. To this end note that $A$-valued $(0,1)$ forms $h_i\bar\partial h_s$ are $C^\infty$ on $E(S,\beta G)$ and have compact supports in $M_a$. Then according to Theorem \ref{dbar} the required solutions $b_{is}$ of the above equations exist (they are of class $C^\infty$ on $M_a$ and holomorphic in an open neighbourhood of $M_s$).

The proof of the theorem is complete.
\end{proof}
\begin{proof}[Proof of Theorem \ref{th3}]
Let $M(A)$ be the maximal ideal space of a commutative unital complex Banach algebra $A$. We will work with $\mathcal O(M(H^\infty);A)$ instead of $H_{\rm comp}^\infty(A)$, see Proposition \ref{sheaf}. 

For $f\in \mathcal O(M(H^\infty);A)$ we define 
\[
\hat f(z;\xi):=\xi (f(z)),\quad z\in M(H^\infty),\ \xi\in M(A).
\]
Since
\[
\sup_{(z,\xi)\in M(H^\infty)\times M(A)}|\hat f(z;\xi)|\le \max_{z\in M(H^\infty)}\|f(z)\|_A,
\]
$\, \hat\, : \mathcal O(M(H^\infty);A) \to \ell_\infty(M(H^\infty)\times M(A))$ is a nonincreasing-norm morphism of algebras.
Let us show that each $\hat f\in C(M(H^\infty)\times M(A))$. 
Indeed, if a net $\{z_\alpha\}\subset M(H^\infty)$ converges to $z$, then $\lim_\alpha \|f(z_\alpha)- f(z)\|_A=0$ by continuity of $f$.
Hence, if a net $\{(z_\alpha,\xi_\alpha)\}\subset M(H^\infty)\times M(A)$ converges to $(z,\xi)$, then 
$$
\begin{array}{l}
\displaystyle
\limsup_\alpha |\hat f(z_\alpha ;\xi_\alpha)-\hat f(z;\xi)|\leq\limsup_\alpha |\hat f(z;\xi_\alpha)-\hat f(z_\alpha ;\xi_\alpha)|\\
\displaystyle +\limsup_\alpha |\hat f(z ;\xi_\alpha)- \hat f(z;\xi)|\le \limsup_\alpha \|f(z_\alpha)- f(z)\|_A+0=0.
\end{array}
$$
(The second term equals zero because $\{\xi_\alpha\}$ converges to $\xi$.)

So the image of $\, \hat\, $ is a subalgebra of $C(M(H^\infty)\times M(A))$. Therefore the operator adjoint to $\, \hat\, $ determines a continuous map 
\[
\iota : M(H^\infty)\times M(A)\to M(\mathcal O(M(H^\infty);A)).
\]

Next, show that $\iota$ is injective. Indeed,
if $\iota((z_1,\xi_1))=\iota((z_2,\xi_2))$, then $\xi_1 (f(z_1))=\xi_2 (f(z_2))$ for all $f\in \mathcal O(M(H^\infty);A)$. Choosing here
$f$ a constant function (equals an element of $A$) we obtain that $\xi_1(a)=\xi_2(a)$ for all $a\in A$. So, $\xi_1=\xi_2=:\xi$. Now, if $z_1\ne z_2$, then for any $a\in A$ there exists a function in $\mathcal O(M(H^\infty);A)$ such that $f(z_1)=a$ and $f(z_2)=0$. This implies that $\xi(a)=\xi(0)=0$ for all $a\in A$ and contradicts nontriviality of $\xi$. Hence, $z_1=z_2$ as well and so $\iota$ is an embedding.

To show that $\iota(M(H^\infty)\times M(A))= M(\mathcal O(M(H^\infty);A))$ it suffices to prove the following corona theorem (for similar arguments see, e.g., \cite[Ch.~V, Th.~1.8]{Ga}):

Suppose that $f_1,\dots, f_m\in \mathcal O(M(H^\infty);A)$ and
\begin{equation}\label{eq5.4}
\max_{1\le j\le m}|\hat f_j(z;x)|>0,\quad (z,x)\in M(H^\infty)\times M(A).
\end{equation}
Then there exist $g_1,\dots, g_m\in \mathcal O(M(H^\infty);A)$ such that
\begin{equation}\label{eq5.5}
f_1 g_1+\cdots + f_m g_m=1.
\end{equation}

Condition \eqref{eq5.4} implies, in particular, that for a fixed $z\in M(H^\infty)$ the elements $f_1(z),\dots, f_m(z)\in A$ do not belong to a maximal ideal of $A$. Therefore the ideal generated by these elements contains $1$, that is, there exist $g_1,\dots g_m\in A$ such that
\[
\sum_{j=1}^m f_j(z)g_j=1.
\]
From here and continuity of functions $f_j$ on the compact set $M(H^\infty)$ we obtain that there exists an open neighbourhood $U\subset M(H^\infty)$ of $z$ such that 
\[
\|1-h(w)\|_A\le\frac 12\quad\text{for all}\quad w\in U,
\]
where $h(w):=\sum_{j=1}^m f_j(w)g_j$,\ $w\in U$. 

This inequality implies that
\[
h^{-1}=(1-(1-h))^{-1}=\sum_{i=0}^\infty (1-h)^i ,
\]
where the series on the right converges uniformly on $U$. Thus $h, h^{-1}\in \mathcal O(U;A)$ and we have
\[
\sum_{j=1}^m f_j\cdot (g_jh^{-1})=1\quad\text{on}\quad U;
\]
here all $g_j h^{-1}\in \mathcal O(U;A)$. 

Hence, $1$ belongs to the ideal of $\mathcal O(U;A)$ generated by $f_1|_{U}\dots, f_m|_{U}$.

Taking an open cover of $M(H^\infty)$ by such sets $U$ and applying Theorem \ref{th2} we obtain that the function $1$ belongs to the ideal of $\mathcal O(M(H^\infty);A)$ generated by $f_1,\dots, f_m$. 

This shows that $\iota: M(H^\infty)\times M(A)\to M(\mathcal O(M(H^\infty);A))$ is a homeomorphism and completes the proof of the theorem. 
\end{proof}
\begin{proof}[Proof of Theorem \ref{square}]
Gluing together points not separated by $A$, without loss of generality we may assume that $A$ separates points of $X$. Then, since $A$ is self-adjoint, it coincides with $C(X)$ by the Stone-Weierstrass theorem.

Assume that $R_{U,H}\Subset R_{U',H}\Subset M_a$ are coordinate charts. First, we prove that 
\begin{equation}\label{Bezout}
f=\sum_{j=1}^m h_j f_j\quad\text{on}\quad R_{U,H} 
\end{equation}
for some $h_j\in \mathcal O(R_{U,H};A)$, $1\le j\le m$. 

Without loss of generality we will identify $R_{U,H}$ with $\Di\times H$ and $R_{U',H}$ with $\Di_r\times H$ for some $r>1$. 

We set $Z:=\{(z,\xi,x)\in \Di_r\times H\times X\, ;\, f_1(z,\xi,x)=\cdots=f_m(z,\xi,x)=0\}$ and define 
$$
\varphi_j:=\left\{
\begin{array}{ccc}
\displaystyle
\frac{\bar f_j}{\sum_{k=1}^m |f_k|^2}&\text{on}&(\Di_r\times H\times X)\setminus Z\\
\\
0&\text{on}&Z.
\end{array}
\right.
$$
Condition (A) of the theorem implies that each $\varphi_j$ is continuous on $\Di_r\times H\times X$. Indeed, $\varphi_j$ is continuous on $(\Di_r\times H\times X)\setminus Z$ and if a net $\{z_\alpha\}\subset (\Di_r\times H\times X)\setminus Z$ converges to a point $z\in Z$, then condition (A) guarantees that  $\lim_\alpha \varphi_j(z_\alpha)=0$.

Next, we define functions $G_{jk}$ on $\Di_r\times H\times X$ by the formulas:
$$
G_{jk}:=\left\{
\begin{array}{ccc}
\displaystyle
f\varphi_j\frac{\partial\varphi_k}{\partial\bar z}&\text{on}&(\Di_r\times Y_2)\setminus Z\\
\\
0&\text{on}&(\Di_r\times Y_1)\cup ((\Di_r\times Y_2)\cap Z),
\end{array}
\right.
$$
where $Y_1:=\{(\xi,x)\in H\times X\, ;\, f_1(\cdot,\xi,x)=\cdots=f_m(\cdot,\xi,x)=0\}$, $Y_2:=(H\times X)\setminus Y_1$.

From condition (A) and Cauchy inequalities for derivatives of bounded holomorphic functions on $\Di_r$ we obtain on each
$\Di_s\times Y_2$, $1\le s<r$,
\[
|G_{jk}|=\left|\frac{f\cdot\bar f_j\bigl(\sum_{\ell=1}^m f_\ell(\bar f_\ell '\bar f_k-\bar f_\ell \bar f_k ') \bigr)}{\bigl(\sum_{\ell=1}^m |f_k|^2\bigr)^3}\right|\le \frac{c_1|f|}{\bigl(\sum_{\ell=1}^m |f_k|^2\bigr)^{3/2}}\le c_2\cdot \tilde\omega\left(\max_{1\le \ell\le m} |f_k|^2\right),
\]
for $\tilde\omega(t):=\frac{\omega(t)}{t^3}$ and some constants $c_1, c_2$ depending on $\max_{1\le \ell\le m}\{\sup_{M(H^\infty)\times X}|f_j|\}$, $m$, $\frac{1}{r-s}$ and
$c$ from condition (A).

Since $\lim_{t\to 0^+}\tilde\omega(t)=0$, this inequality and the arguments similar to those used for $\varphi_j$ show that each $G_{jk}$ is continuous on $\Di_r\times H\times X$.

Further, to obtain a holomorphic solution of Bezout equation \eqref{Bezout} we must solve equations (cf. \cite[Ch.~VIII, Th.~2.3]{Ga}) 
\begin{equation}\label{eq5.6}
\frac{\partial b_{jk}(z,\xi,x)}{\partial\bar z}=G_{jk}(z,\xi,x)\quad \text{on}\quad \Di\times H\times X.
\end{equation}
This can be done by the standard formula
\begin{equation}\label{poin}
b_{jk}(z,\xi,x):=\frac{1}{2\pi i}\int\int_{\Co}\frac{\rho(w)G_{jk}(w,\xi,x)}{w-z}dw\wedge d\bar w,\quad (z,\xi,x)\in \bar\Di\times H\times X , 
\end{equation}
where $\rho$ is a real $C^\infty$ function equals $1$ on $\Di$ and $0$ on $\Co\setminus\Di_r$.
 
Since each $G_{jk}$ is continuous on $\Di_r\times H\times X$, \eqref{poin} (rewritten in polar coordinates $w:=z+re^{i\theta}$) implies that each $b_{jk}$ is continuous on $\bar\Di\times H\times X$. Also,
observe that $G_{jk}(\cdot,\xi, x)$ may be either $0$ (for $(\xi,x)\in Y_1$) or a complex analytic function (for $(\xi,x)\in Y_2$, cf. \cite{Ga}).  Thus $b_{jk}(\cdot,\xi, x)$ is of class $C^\infty$ for each $(\xi,x)\in H\times X$. In particular, as in \cite{Ga}, for each $(\xi,x)\in H\times X$ functions
\[
g_j(\cdot,\xi, x)= f(\cdot,\xi, x)\psi_j(\cdot, \xi, x)+\sum_{k=1}^m (b_{jk}(\cdot,\xi,x)- b_{kj}(\cdot,\xi,x)) f_k(\cdot,\xi, x)
\]
belong to $\mathcal O(\Di)\cap C(\bar\Di)$ and satisfy $\sum_{j=1}^m g_jf_j= f$ on $\Di\times H\times X$. Moreover, all $g_j$ are continuous on $\bar\Di\times H\times X$.  So they can be regarded as holomorphic functions on $\Di\times H$  continuous on $\bar\Di\times H$ (recall that $H$ is clopen) with values in $C(X):=A$.

Hence, we have proved that for each $R_{U,H}\Subset M_a$ the restriction
$f|_{R_{U,H}\times X}$ belongs to the ideal of $\mathcal O(R_{U,H};A)$ generated by $f_1|_{R_{U,H}\times X},\dots, f_m|_{R_{U,H}\times X}$. 

Finally, due to condition (B), for each $w\in M_s$ function $f(w,\cdot)\in C(X)$ belongs to the ideal generated by $f_1(w,\cdot),\dots, f_m(w,\cdot)$, cf. arguments after \eqref{eq5.5}. As in the proof of Theorem \ref{th3} we deduce from here that for each $w\in M_s$ there exists its open neighbourhood $U_w\subset M(H^\infty)$ such that $f|_{U_w\times X}$ belongs to the ideal of $\mathcal O(U_w;A)$ generated by 
$f_1|_{U_w\times X},\dots, f_m|_{U_w\times X}$. 

Taking a finite open cover of $M(H^\infty)$ by sets $U_w$, $w\in M_s$, and $R_{U,H}\subset M_a$ and applying
Theorem \ref{th2} we conclude that $f$ belongs to the ideal of $S(A)$ generated by $f_1,\dots, f_m$.

The proof of the theorem is complete.
\end{proof}

\begin{proof}[Proof of Corollary \ref{cor1}]
By definition \[
S_N(H^\infty)=S(S_{N-1}(H^\infty))\, (:=S(H^\infty; S_{N-1}(H^\infty))).
\]
Hence, from Theorem \ref{th3} we obtain 
\[
M(S_N(H^\infty))=M(H^\infty)\times M(S_{N-1}(H^\infty))=\cdots=M(H^\infty)^N. 
\]
\end{proof}

\section{Proof of Theorems \ref{th4} and \ref{th5}}
\begin{proof}[Proof of Theorem \ref{th4}]
By Theorem \ref{th3} the maximal ideal space of $H_{\rm comp}^\infty(A)$ can be naturally identified with $M(H^\infty)\times M(A)$.

Next, conditions of the theorem imply that the image $\hat F=(\hat f_{ij})$ of the matrix $F$ under the Gelfand transform $\, \hat\, $ satisfies 
\[
\sum_{i=1}^\ell |\hat h_i(x,y))|\ge\delta\quad\text{for all}\quad (x,y)\in M(H^\infty)\times M(A),
\]
where $\hat h_i,\dots, \hat h_{\ell}$ is the family of minors of order $k$ of $\hat F$. This inequality allows to apply \cite[Th.~3]{Lin} asserting that in order to prove the result it suffices to extend the matrix $\hat F$ up to an invertible one in the category of continuous matrix functions on $M(H^\infty)\times M(A)$, i.e. to find an $n\times n$
matrix $G = (g_{ij}),\ g_{ij}\in C(M(H^\infty)\times M(A))$, so that $g_{ij} = \hat f_{ij}$ for $1\le j\le k$, $1\le i\le n$, and
${\rm det}\,G = 1$. 

Note that the matrix $\hat F$ determines a trivial subbundle $\xi$ of complex rank $k$ in
the trivial vector bundle $\theta^n:=\bigl(M(H^\infty)\times M(A)\bigr)\times\Co^n$ on $M(H^\infty)\times M(A)$. Let $\eta$ be an additional to $\xi$ subbundle of $\theta^n$, i.e., $\xi\oplus\eta=\theta^n$. We show that $\eta$ is topologically trivial. Then a trivialization of $\eta$ given by global continuous sections $s_1,\dots, s_{n-k}$ determines the required extension $G$ of $\hat F$.

In what follows $\theta^\ell$ stands for the trivial rank $\ell$ complex vector bundle on a compact topological space.
\begin{Lm}\label{complet}
Let $\vartheta$ be a rank $n-k$ complex vector bundle on a compact topological space $Y$ satisfying $\theta^k\oplus\vartheta=\theta^n$. Assume that
$n-k\ge \lfloor\frac{s}{2}\rfloor$, where $s\ge {\rm dim}\, Y$. Then $\vartheta\cong\theta^{n-k}$.
\end{Lm}
\begin{proof}
First, we reduce the statement of the lemma to the case of bundles on a metrizable space. Indeed, according to \cite{M} one presents $Y$ as the inverse limit of a limiting system $\{Y_\alpha\, ;\, \pi\}_{\alpha\in \Lambda}$, where $Y_\alpha$ are metrizable compact spaces of dimension $\le s$. Then by a well-known theorem about
continuous maps of inverse limits of compact spaces (see, e.g.,
\cite{EilSte}) and the fact that all complex vector bundles of rank
$n$ on $Y$ can be obtained as pullbacks of the universal bundle
$EU(n)$ on the classifying space $BU(n)$ of the unitary group
$U(n)\subset GL_n(\Co)$ under some continuous maps $Y\to BU(n)$
(see, e.g., \cite{Hus}), for the bundle $\vartheta$ there is $\alpha_0\in\Lambda$ and a complex vector
bundle $\vartheta_{\alpha_0}$ on $Y_{\alpha_0}$ such that the pullback
$\pi_{\alpha_0}^*\vartheta_{\alpha_0}$ is isomorphic to $\vartheta$.  Moreover, since $\theta^k\oplus\vartheta=\theta^n$, increasing $\alpha_0$, if necessary, we may assume without loss of generality that $\theta^k\oplus\vartheta_{\alpha_0}=\theta^n$ on $Y_{\alpha_0}$ (this follows, e.g., from \cite[Lm.~1]{Lin}.) If we show that under the conditions on $d, n, k$ the bundle $\vartheta_{\alpha_0}\cong\theta^{n-k}$, then the same will be true for the bundle $\vartheta$. 

Thus without loss of generality we may assume that $Y$ is metrizable. Further, using the classical Freudenthal theorem we can present $Y$ as the inverse limit of a sequence of compact polyhedra of dimension $\le s$. Applying arguments as above we may assume without loss of generality that $Y$ is a compact polyhedron of dimension $\le s$. But then under the conditions of the lemma the required statement (i.e., $\vartheta\cong\theta^{n-k}$) follows directly from \cite[Ch.~9, Th.~1.5]{Hus}.
\end{proof}
To apply the lemma observe that according to the hypothesis of the theorem $M(A)$ is the inverse limit of a limiting system $\{M_\alpha\, ;\,\pi\}_{\alpha\in\Lambda}$, where each $M_\alpha$ is homotopically equivalent to a metrizable compact space $X_\alpha$ with ${\rm dim}\, X_\alpha\le d$. Therefore
$M(H_{\rm comp}^\infty(A))=M(H^\infty)\times M(A)$ is the inverse limit of the system $\{M(H^\infty)\times M_\alpha\, ;\,{\rm id}\times\pi\}_{\alpha\in\Lambda}$. Thus as in the proof of Lemma \ref{complet} we obtain that in order to prove that $\eta\cong\theta^{n-k}$ it suffices to prove a similar statement for bundles $\eta_\alpha$ on $M(H^\infty)\times M_\alpha$ satisfying $\eta_\alpha\oplus\theta^k=\theta^n$.
Further, $M(H^\infty)\times M_{\alpha}$ is homotopically equivalent to $M(H^\infty)\times X_\alpha$ and therefore each
$\eta_{\alpha}$ is isomorphic to the pullback (under the map establishing the homotopy equivalence) of some bundle $\tilde\eta_{\alpha}$ on $M(H^\infty)\times X_\alpha$ satisfying $\theta^k\oplus \tilde\eta_{\alpha}=\theta^n$. But due to \cite{S1} ${\rm dim}\, M(H^\infty)=2$; hence,
${\rm dim}\, M(H^\infty)\times X_\alpha\le d+2$. Applying Lemma \ref{complet} together with conditions of the theorem to $\tilde\eta_\alpha$ we obtain its triviality. This implies triviality of $\eta_\alpha$ and then of $\eta$.

The proof of the theorem is complete.
\end{proof}

\begin{proof}[Proof of Theorem \ref{th5}]
According to assumptions of the theorem $M(A)$ is the inverse limit of a limiting system $\{X_\alpha\, ;\,\pi\}_{\alpha\in\Lambda}$, where each $X_\alpha$ is a metrizable contractible compact space. Therefore
$M(H_{\rm comp}^\infty(A))=M(H^\infty)\times M(A)$ is the inverse limit of the system $\{M(H^\infty)\times X_\alpha\, ;\,{\rm id}\times\pi\}_{\alpha\in\Lambda}$. Since ${\rm dim}\, M(H^\infty)=2$ and $H^2(M(H^\infty);\Z)=0$ (see \cite{S1}), any finite rank complex vector bundle on $M(H^\infty)$ is topologically trivial. Thus, since $M(H^\infty)\times X_\alpha$ is homotopically equivalent to $M(H^\infty)$, the same is valid for finite rank complex vector bundles on that space. From here as in the proof of Theorem \ref{th3} we obtain that any finite rank complex vector bundle on $M(H^\infty)\times M(A)$ is topologically trivial. Now the desired result follows from \cite[Th.~1.3]{BS}
\end{proof}

\sect{Proof of Theorem \ref{approx}}
\subsection{} 
Assume that $R_{U,H}\Subset R_{U',H}\Subset M_a$ are coordinate charts. We will use the following auxiliary result.
\begin{Lm}\label{Ltay}
For any function $f\in\mathcal O(R_{U',H};B)$ its restriction
$f|_{R_{U,H}}$ belongs to $H^\infty(R_{U,H})\otimes B$. 
\end{Lm}
\begin{proof}
Without loss of generality we will identify $R_{U,H}$ with $\Di\times H$ and $R_{U',H}$ with $\Di_r\times H$ for some $r>1$. Since, $H$ is clopen and $f$ is $B$-valued continuous on the compact set $\bar\Di_s\times H$ for some $1<s<r$, we can regard $\tilde f:=f|_{\Di_s\times H}$ as a function from $H_{\rm comp}^\infty(\Di_s;C(H;B))\cap C(\bar\Di_s;C(H;B))$. Applying to $\tilde f$ the Cauchy integral formula and then decomposing the Cauchy kernel we obtain 
\begin{equation}\label{tay1}
f(z,\xi)=\sum_{j=0}^\infty a_j(\xi)z^j,\qquad (z,\xi)\in \Di\times H,
\end{equation}
where all $a_j\in C(H;B)$ and the series converges uniformly to $f$, i.e., 
\begin{equation}\label{tay2}
\lim_{N\to\infty}\sup_{(z,\xi)\in\Di\times H}\left\|f(z,\xi)-\sum_{j=0}^N a_j(\xi)z^j\right\|_B=0.
\end{equation}

Further, since $H$ is a compact Hausdorff space, $C(H)$ has the approximation property; in particular, $C(H,B)=C(H)\otimes B$, see Theorem \ref{th1}.
From here and \eqref{tay1}, \eqref{tay2} we obtain that $f|_{\Di\times H}\in H^\infty(\Di\times H)\otimes B$.
\end{proof}

\subsection{} Now we prove Theorem \ref{approx}.
\begin{proof}

Let us prove the first statement of the theorem. Let $U\Subset V\subset M_a$ be open subsets. Choose open $W\Subset V$ containing $\bar U$ and
coordinate charts $R_{U_j,H_j}$, $R_{U_j',H_j}$, $1\le j\le m$, such that $R_{U_j,H_j}\Subset R_{U_j',H_j}\Subset V$ for all $j$, and
$\mathcal R:=(R_{U_j,H_j})_{1\le j\le m}$ is an open cover of $\bar W$. Let $f\in\mathcal O(V;B)$. According to Lemma \ref{Ltay} for any $n\in\N$ and $1\le j\le m$ there are functions $f_{j,n}\in H^\infty(R_{U_j,H_j})\otimes B$ of the form
\[
f_{j,n}:=\sum_{s=1}^{n_j}b_{js,n}f_{js,n},\quad\text{where all}\quad b_{js, n}\in B,\ f_{js, n}\in H^\infty(R_{U_j,H_j}),
\]
such that
\[
\sup_{R_{U_j,H_j}}\|f-f_{j,n}\|_B\le\frac{1}{2n}.
\]

By $B_n\subset B$ we denote the finite-dimensional vector subspace generated by all $b_{js,n}$, for all possible $j$ and $s$. 
Consider a cocycle $\{c_{ij,n}\}\in Z^1(\mathcal R ;B_n)$ defined by the formulas
\[
c_{ij,n}:=f_{i,n}-f_{j,n}\quad\text{on}\quad R_{U_i,H_i}\cap R_{U_j,H_j}\ne\emptyset.
\]
Then for all $i,j$
\[
\sup_{R_{U_i,H_i}\cap R_{U_j,H_j}}\|c_{ij,n}\|_B\le\frac{1}{n}.
\]

Next, we introduce the Banach space $\tilde B$ of sequences $v=(v_1,v_2,\dots)$ such that $v_n\in B_n$, $n\in\N$, with norm 
\[
\|v\|_{\tilde B}:=\left(\sum_{\ell=1}^\infty \|v_\ell\|_B^2\right)^{1/2}.
\]

Clearly, $c_{ij}:=(c_{ij,1}, c_{ij,2},\dots)$ is a holomorphic function on $R_{U_i,H_i}\cap R_{U_j,H_j}$ with values in $\tilde B$. Thus
$c=\{c_{ij}\}$ is a holomorphic $1$-cocycle on the cover $\mathcal R$ with values in $\tilde B$.

Let $U'$ be open containing $\bar U$ and such that $\bar U'$ belongs to the union of all $R_{U_j,H_j}$.
We set 
$$
U_*:=M(H^\infty)\setminus\bar U'.
$$
Then $U_*$ together with all  $R_{U_j,H_j}$ form a finite open cover of $M(H^\infty)$.
As in subsection~4.2 consider a refinement of this cover $(W_j)_{1\le j\le k}$ and subordinate to the cover $(W_j')_{1\le j\le k}$, $W_j':=q^{-1}(W_j)$, of $E(S,\beta G)$ a $C^\infty$ partition of unity $\{\varphi_j\}_{1\le j\le k}$. Without loss of generality we may assume that each $W_r':=W_r$ with $1\le r\le s$ belongs to one of
$R_{U_j,H_j}$ and others $W_j$ are subsets of $U_*$ only. 

Observe that ${\mathcal W}=(W_j)_{1\le j\le r}$ is a cover of $\bar U$.
(For otherwise, there exists $W_j\subset U_*$ such that $W_j\cap\bar U\ne\emptyset$, a contradiction.) By $\tilde c=(\tilde c_{ij})\in Z^1({\mathcal W};\tilde B)$ we denote the restriction of the cocycle $c$ to the cover $\mathcal W$. Next, for $1\le i\le r$ we define
\[
h_i=\sum_\ell\varphi_\ell\tilde c_{i\ell}\quad\text{on}\quad W_i;
\]
here the sum is taken over all $\ell$ for which $W_\ell\cap W_i\ne\emptyset$.

Suppose that $U''$ is an open set containing $\bar U$ and such that $\bar U''\Subset U'$. Then $\{\varphi_j|_{U''}\}$ is a $C^\infty$ partition of unity subordinate to the cover $(W_j\cap U'')_{1\le j\le r}$ of $U''$. In particular, $\{h_i|_{U''}\}$ is a resolution of the cocycle $\tilde c|_{U''}$. Then the formulas 
\[
\omega:=\bar\partial (h_i|_{U''})\quad\text{on}\quad U''\cap W_i
\]
define a $\tilde B$-valued $C^\infty$ $(0,1)$-form on $U''$.

Applying Lemma \ref{rest} we find a $\tilde B$-valued $C^\infty$ function $g$ on $U''$ such that $\bar\partial g=\omega$ on $U$.

Next, for each $W_i\cap U\ne\emptyset$ we set
\[
c_i=h_i|_U-g|_{W_i\cap U}.
\]
Then $c_i\in H_{\rm comp}^\infty(W_i\cap U; \tilde B)$ and 
\[
c_i-c_j=\tilde c_{ij}|_{U}\quad\text{on}\quad W_i\cap W_j\cap U\ne\emptyset.
\]
By definition, $c_i=(c_{i,1}, c_{i,2},\dots)$, where $c_{i,n}\in H^\infty(W_i\cap U)\otimes B_n$ and 
\[
\lim_{n\to\infty}\sup_{W_i\cap U}\|c_{i,n}\|_B=0.
\]
Finally, we define functions $f_n$ on $U$ by the formulas
\[
f_n=f_{i,n}-c_{i,n}\quad\text{on}\quad W_i\cap U\ne\emptyset.
\]
(Since $(f_{i,n}-c_{i,n})-(f_{j,n}-c_{j,n})=\tilde c_{ij}|_{U}-\tilde c_{ij}|_{U}=0$ on $W_i\cap W_j\cap U\ne\emptyset$, each $f_n$ is well-defined.) According to our construction $f_n\in H^\infty(U)\otimes B_n$ and 
\[
\lim_{n\to\infty}\sup_{W_i\cap U}\|f-f_{n}\|_B=0.
\]
Hence, $f\in H^\infty(U)\otimes B$.

This completes the proof of the first statement of the theorem.

\medskip

The proof of the second statement is similar, so we will briefly describe it.

According to the first part of the theorem and its assumptions, for a function $f\in\mathcal O(M(H^\infty);B)$ there exists an open cover $(U_j)_{1\le j\le m}$ of $M(H^\infty)$ such that $f|_{U_i}\in H^\infty(U_i)\otimes B$ for all $i$. Then we repeat word-for-word the proof of the first part of the theorem with $R_{U_j,H_j}$ replaced by $U_j$, to construct a holomorphic 1-cocycle $c$ on the cover $(U_j)_{1\le j\le m}$ with values in $\tilde B$. By Theorem \ref{main} passing to a suitable refinement $(W_j)$ of this cover we can resolve the restriction of $c$ to $(W_j)$. Then we conclude as in the proof of the first part of the theorem (with $W_i\cap U$ replaced by $W_i$).

The converse to this statement follows from Theorem \ref{th1}: if $H^\infty$ has the approximation property, then $\mathcal O(M(H^\infty); B)=H^\infty\otimes B$.
\end{proof}

\sect{Further Results}
Let $R$ be a Caratheodory hyperbolic Riemann surface and $A$ be a commutative unital complex Banach algebra. By $H_{\rm comp}^\infty (R;A)$ we denote the Banach algebra of holomorphic functions on $\Di$ with relatively compact images in $A$ equipped with norm $\|f\|:=\sup_{z\in R}\|f(z)\|_A$, $f\in H_{\rm comp}^\infty (R;A)$. Let $p:\Di\to R$ be the universal covering of $R$. The fundamental group $\pi_1(R)$ acts of $\Di$ by M\"{o}bius transformations. Assume that $R$ satisfies:
\begin{itemize}
\item[(P)]
There exists a function $h\in H^\infty$ such that $\sup_{z\in\Di}\left(\sum_{g\in\pi_1(R)}|h(gz)|\right)<\infty$ and $\sum_{g\in\pi_1(R)}h(gz)=1$ for all $z\in\Di$.
\end{itemize}
This $h$ determines a linear continuous map $P_A: H_{\rm comp}^\infty (A)\to H_{\rm comp}^\infty (R;A)$ such that
\[
P_A(f_1\cdot p^*f_2)=P_A(f_1)\cdot f_2\quad\text{for all}\quad  f_1\in  H_{\rm comp}^\infty (A),\ f_2\in H_{\rm comp}^\infty (R;A),
\]
given by the formula
\begin{equation}\label{proj}
P_A(f):=p_*\left(\sum_{g\in\pi_1(R)}h(gz)f(gz)\right),\quad f\in H_{\rm comp}^\infty (A);
\end{equation}
here $p_*: p^*\left(H_{\rm comp}^\infty (R;A)\right)\to  H_{\rm comp}^\infty (R;A)$ is inverse to the pullback map $p^*$.

Existence of such $h$ for $R$ a finite bordered Riemann surface was established in \cite{Fo}, for $R$ a homogeneous Denjoy domain in \cite{C2} (see also \cite{JM} for a more general setting), and for $R$ a subdomain of an unbranched covering $R'$ of a finite bordered Riemann surface such that inclusion $R\hookrightarrow R'$ induces a monomorphism of the fundamental groups in \cite{Br2}.
In all these cases taking the pullback by $p$ of functions from $H_{\rm comp}^\infty (R;A)$, then solving the corresponding problem in $H_{\rm comp}^\infty (A)$ and applying to the solution the map $P_A$ we obtain analogs of Theorems \ref{runge}, \ref{th2}, \ref{th3} on $R$ (with $H^\infty$ replaced by $H^\infty(R)$ and $H_{\rm comp}^\infty (A)$ replaced by $H_{\rm comp}^\infty (R;A)$ in the original versions of these theorems). In fact, similar results are valid on a Riemann surface $R$ of finite type (since it is biholomorphic to a bordered Riemann surface with finitely many points removed). In this case arguing as in \cite{Br2} one obtains that ${\rm dim}\, M(H^\infty(R))=2$. In particular, analogs of Theorems \ref{th4} and \ref{th5} are valid on $R$.

\end{document}